\documentclass{article}
\usepackage{arxiv}
\usepackage{microtype}

\usepackage{graphicx}
\usepackage{subcaption}
\usepackage{booktabs} % for professional tables

\usepackage{hyperref}

\usepackage{amsmath}
\usepackage{amssymb}
\usepackage{amsthm}
\usepackage{xcolor}
\usepackage{paralist}

\usepackage{natbib}
\usepackage{algorithm}
\usepackage{algorithmic}

\newtheorem{theorem}{Theorem}[section]
\newtheorem{lemma}[theorem]{Lemma}

\newtheorem{definition}{Definition}[section]

\newif\ifcomm
\commtrue

\newif\iflong
\longtrue
\newcounter{assumption}%[section]
\renewcommand{\theassumption}{A\arabic{assumption}}

\newcommand{\beq}{\begin{equation}}
\newcommand{\eeq}{\end{equation}}

\ifcomm
   \newcommand\comm[1]{\textcolor{blue}{ #1}}
   \newcommand{\mtodo}[2]{\todo{{\bf #1}: #2}} % To add comments into the text; the first argument is "who", the second is "what"
   \def\here#1{{\bf $\langle\langle$#1$\rangle\rangle$}}

\else
   \newcommand\comm[1]{}
   \newcommand{\mtodo}[2]{}
   \def\here#1{}
\fi

\def\be{\begin{equation}}
\def\ee{\end{equation}}

\def\y{\mathbf{y}}

\def\x{\mathbf{x}}
\def\z{\mathbf{z}}

\def\M{\mathbf{M}}
\def\1{\mathbf{1}}

\def\u{\mathbf{u}}
\def\v{\mathbf{v}}
\def\w{\mathbf{w}}

\def\Z{\mathbf{Z}}

\def\X{\mathbf{X}}
\def\U{\mathbf{U}}

\def\l{\boldsymbol{\ell}}
\def\PP{\mathbb{P}}

\newcommand{\definedas}{\overset{\underset{\mathrm{def}}{}}{=}}

\title{Towards Practical Mean Bounds for Small Samples}

\author{My Phan\thanks{College of Information and Computer Sciences, University of Massachusetts Amherst.}~\thanks{Correspondence to: myphan@cs.umass.edu.}~,  Philip S. Thomas\thanks{College of Information and Computer Sciences, University of Massachusetts.}~  and Erik~Learned-Miller\footnotemark[1]}
\begin{document}
\maketitle
\begin{abstract}
Historically, to bound the mean for small sample sizes, 
practitioners have had to choose between using methods with unrealistic assumptions about the unknown distribution (e.g., Gaussianity) and methods like Hoeffding's inequality that use weaker assumptions but produce much looser (wider) intervals. 
In 1969, \citet{Anderson1969}  proposed a mean confidence interval strictly better than or equal to Hoeffding's whose only assumption is that the distribution's support is contained in an interval $[a,b]$. For the first time since then, we present a new family of bounds that compares favorably to Anderson's. We prove that each bound in the family has {\em guaranteed coverage}, i.e., it holds with probability at least $1-\alpha$ for all distributions on an interval $[a,b]$. Furthermore, one of the bounds is tighter than or equal to Anderson's for all samples. In simulations, we show that for many distributions, the gain over Anderson's bound is substantial. 

\end{abstract}

\section{Introduction}
\label{sec:intro}

In this work,\footnote{{This is an extended work of our paper \citep{phan2021}.}} we revisit the classic statistical problem of defining a confidence interval on
the mean $\mu$ of an unknown distribution with CDF $F$ from an i.i.d.~sample $\mathbf{X}=X_1, X_2, \dotsc,X_n$, and the closely related problems of producing upper or lower confidence bounds on the mean. For simplicity, we focus on upper confidence bounds (UCBs), but the development for lower confidence bounds and confidence intervals is similar. 

To produce a non-trivial UCB, one must make assumptions about $F$, such as finite variance, sub-Gaussianity, or that its support is contained on a known interval $[a,b]$. We adopt this last assumption, working with distributions whose support is known to fall in an interval $[a,b]$. For UCBs, we refer to two separate settings, the {\em one-ended support} setting, in which the distribution is known to fall in the interval $[-\infty,b]$, and the {\em two-ended support} setting, in which the distribution is known to fall in an interval $[a,b]$, where $a>-\infty$ and $b<\infty$. 
%, and without loss of generality, assume that the support of $F$ is contained in the interval $[0,1]$. 
%The bounds we present can be trivially extended to semi-infinite interval. 

A UCB has {\em guaranteed coverage} for a set of distributions $\mathcal{F}$ if, for all sample sizes $1\leq n \leq \infty$, for all confidence levels $1-\alpha\in (0,1)$, and for all distributions $F\in \mathcal{F}$,  the bound $\mu^{1-\alpha}_{\tt{upper}}$ satisfies
\begin{equation}
\label{eq:valid}
Prob_F[\mu\leq \mu^{1-\alpha}_{\tt{upper}}(X_1,X_2,...,X_n)] \geq 1-\alpha,
\end{equation}
where $\mu$ is the mean of the unknown distribution $F$.

Among bounds with guaranteed coverage for distributions on an interval $[a,b]$, our interest is in bounds with good performance on {\em small sample sizes}. The reason is that, for `large enough' sample sizes, excellent bounds and confidence intervals already exist. In particular, the confidence intervals based on Student's $t-$statistic~\citep{student1908probable} are satisfactory in terms of coverage and accuracy for most practitioners, given that the sample size is greater than some threshold.\footnote{An adequate sample size for the Student's $t$ method  depends upon the setting, but a common rule is $n>30$.}

The validity of the Student's $t$ method depends upon the Gaussianity of the sample mean, which, strictly speaking does not hold for any finite sample size unless the original distribution itself is Gaussian. However, for many applications, the sample mean becomes close enough to Gaussian as the sample size grows (due to the effects described by the central limit theorem), that the resulting bounds hold with probabilities close to the confidence level. Such results 
vary depending upon the unknown distribution, but it is generally accepted that a large enough sample size can be defined to cover any distributions that might occur in a given situation.\footnote{An example in which the sample mean is still visibly skewed (and hence inappropriate for use with Student's $t$) even after $n=80$ samples is given for log-normal distributions in the supplementary material.} 
The question is what to do when the sample size is smaller than such a threshold.  

Establishing good confidence intervals on the mean for small samples is an important but often overlooked problem. The $t$-test is widely used in medical and social sciences. Small clinical trials (such as Phase 1 trials), where such tests could potentially be applied, occur frequently in practice~\citep{NAP10078}. In addition, there are several machine learning applications. The sample mean distribution of an importance-weighted estimator is skewed even when the sample size is much larger than $30$, so tighter bounds with guarantees may be beneficial.  Algorithms in Safe Reinforcement Learning \citep{DBLP:conf/aaai/ThomasTG15} use importance weights to estimate the return of a policy and use confidence bounds to estimate the range of the mean. The UCB multi-armed bandit algorithm is designed using the Hoeffding bound - a tighter bound may lead to better performance with guarantees.

In the two-ended support setting, our bounds  provide a new and better option for guaranteed coverage with small sample sizes.\footnote{Code accompanying this paper is available at \url{https://github.com/myphan9/small_sample_mean_bounds}.} At least one version of our bound is tighter (or as tight) for {\em every possible sample} than the bound by Anderson~\citep{Anderson1969}, which is arguably the best existing bound with guaranteed coverage for small sample sizes.
In the limit as $a\rightarrow -\infty$, i.e., the one-ended support setting, this version of our bound is equivalent to Anderson.%\footnote{At the time of submission, we had established that a particular version of our bound was tighter than or equal to Anderson's for both the one-ended and the two-ended settings. Subsequently, \citet{phan2021practical} established that this version of our bound is in fact equivalent to Anderson's for the one-ended setting, but superior for many cases in the two-ended setting. We made minor revisions to the text to incorporate this new information.}

It can be shown from \citet{Learned-MillerThomas2019} that Anderson's UCB is less than or equal to Hoeffding's for \emph{any} sample when $\alpha \le 0.5$, and is strictly less than Hoeffding's when $\alpha \le 0.5$ and  $n \ge 3$. Therefore our bound is also less than or equal to Hoeffding's for \emph{any} sample when $\alpha \le 0.5$, and is strictly better than Hoeffding's inequality  when $\alpha \le 0.5$ and $n \ge 3$. 

Below we review bounds with coverage guarantees, those that do {\em not} exhibit guaranteed coverage, and those for which the result is unknown. 

\subsection{Distribution free bounds with guaranteed coverage} 
Several bounds exist that have guaranteed coverage. These include Hoeffding's inequality~\citep{Hoeffding1963}, Anderson's bound~\citep{Anderson1969}, and the bound due to~\citet{Maurer2009}.

\textbf{Hoeffding's inequality.} For a distribution $F$ on $[a,b]$, Hoeffding's inequality \citep{Hoeffding1963} provides a bound on the probability that the sample mean, $\bar X_n = \frac{1}{n}\sum_{i=1}^n X_i$, will deviate from the mean by more than some amount $t\geq0$:
\begin{equation}
    \label{eq:HoeffdingOriginal}
    \Pr\left (\mu - \bar X_n \leq t\right ) \leq e^{-\frac{2nt^2}{(b-a)^2}}.
\end{equation}
Defining $\alpha$ to be the right hand side of this inequality, solving for $t$ as a function of $\alpha$, and rewriting in terms of $\alpha$ rather than $t$, one obtains  a $1-\alpha$ UCB on the mean of 

\begin{equation}
    b^{\alpha,\text{Hoeffding}}(\X) \definedas \bar X_n + (b-a)\sqrt{\frac{\ln(1/\alpha)}{2n}}.
\end{equation}

{\bf Maurer and Pontil.} One limitation of Hoeffding's inequality is that the amount added to the sample mean to obtain the  UCB scales with the range of the random variable over %divided by
$\sqrt{n}$, which shrinks slowly as $n$ increases. 
%Notice also that the upper confidence bound produced by Hoeffding's inequality only depends on the sample mean. 

Bennett's inequality \citep{bennett1962probability} considers both the sample mean and the sample variance and obtains a better dependence on the range of the random variable \emph{when the variance is known}. 
\citet{Maurer2009} derived a UCB for the variance of a random variable, and suggest combining this with Bennet's inequality (via the union bound) to obtain the following  $1-\alpha$ UCB on the mean: 

    \begin{eqnarray*}
    b^{\alpha,\text{M\&P}}(\X) \definedas 
    \bar X_n + \frac{7(b-a)\ln(2/\alpha)}{3(n-1)}+\sqrt{\frac{2\hat\sigma^2 \ln(2/\alpha)}{n}}.
\end{eqnarray*}

Notice that Maurer and Pontil's UCB scales with the range $(b-a)$, divided by $n$ (as opposed to the $\sqrt n$ of Hoeffding's). %
However, the $\sqrt{n}$ dependence is unavoidable to some extent: Maurer and Pontil's  UCB scales with the sample standard deviation $\hat \sigma$ divided by $\sqrt{n}$. 
As a result, Maurer and Pontil's bound tends to be tighter than Hoeffding's when both $n$ is large and the range of the random variable is large relative to the variance. 
Lastly, notice that Maurer and Pontil's bound requires $n{\geq} 2$ %in order 
for the sample standard deviation to be defined.

{\bf Anderson's bound.} 
\citet{Anderson1969}\footnote{An easier to access and virtually equivalent version of Anderson's work can be found in~\citep{Anderson1969TechReport}.} introduces a bound by defining an `envelope' of equal width that, with high probability, contains the true CDF. The upper and lower extremes of such an envelope define the CDFs with the minimum and maximum attainable means for distributions that fit within the envelope, and thus bound the mean with high probability.\footnote{In his original paper, Anderson also suggests a large family of envelopes, each of which produces a distinct bound. Our simulation results in Section~\ref{sec:simulation} are based on the equal-width envelope, but our theoretical results in Section~\ref{sec:comparison} hold for all possible envelopes.}

In practice, Anderson's bound tends to be significantly tighter than Maurer and Pontil's inequality unless the variance of the random variable is miniscule in comparison to the range of the random variable (and $n$ is sufficiently large). 
However, neither Anderson's inequality nor Maurer and Pontil's inequality strictly dominates the other. 
That is, neither upper bound is strictly less than or equal to the other in all cases. 
However, Anderson's bound \emph{does} dominate Hoeffding's inequality \citep{Learned-MillerThomas2019}.
%\my{ mention here that anderson dominate Hoeffding. Also, do we need to clarify what we mean by dominance?}

%A variety of variations on the theme of Anderson's bound have been proposed,

Some authors have proposed specific envelopes for use with Anderson's technique  \citep{Diouf2005,learned2008probabilistic,RomanoWolf2000}. However, none of these variations are shown to dominate Anderson's original bound. That is, while they give tighter intervals for some samples, they are looser for others.

%Other versions of Anderson's bound, based on different envelopes, have been proposed including those by \citet{RomanoWolf2000}. However, none of these variations are shown to dominate Anderson's original bound.

\textbf{Other bounds. }
\citet{Fienberg1977} proposed a bound for distributions on  a discrete set of support points, but nothing prevents it, in theory, from being applied to an arbitrarily dense set of points on an interval such as $[0,1]$. This bound has a number of appealing properties, and comes with a proof of guaranteed coverage. However, the main drawback is that it is currently computationally intractable, with a computation time that depends exponentially on the number of points in the support set, precluding %Hence, it is currently uncomputable in 
many (if not most) practical applications. 

In an independent concurrent work, \citet{waudbysmith2021estimating} proposed another confidence interval for the mean, which generalizes and improves upon Hoeffding's inequality.
%(It was developed to handle certain special cases that occur frequently in auditing applications.) 

\subsection{Bounds that do not exhibit guaranteed coverage}
Many bounds that are used in practice are known to violate Eq.~(\ref{eq:valid}) for certain distributions. These include the aforementioned Student's $t$ method, and various bootstrap procedures, such as the bias-corrected and accelerated (BCa) bootstrap and the percentile bootstrap. See \citet{Efron1993} for details of these methods. A simple explanation of the failure of bootstrap methods for certain distributions is given by \citet[pages 757--758]{RomanoWolf2000}.
Presumably if one wants guarantees of Eq.~(\ref{eq:valid}), one cannot use these methods (unless one has extra information about the unknown distribution).

\subsection{Bounds conjectured to have guaranteed coverage}
There are at least two known bounds that perform well in practice but for which no proofs of coverage are known. One of these, used in accounting procedures, is the so-called Stringer bound~\citep{Stringer1963}. It is known to violate Eq.~(1) for confidence levels $\alpha > 0.5$~\citep{PapZuijlen1995}, but its coverage for $\alpha<0.5$ is unknown.

A little known bound by~\citet{Gaffke2005} gives remarkably tight bounds on the mean, but has eluded a proof of guaranteed coverage. This bound was recently rediscovered by~\citet{Learned-MillerThomas2019}, who do an empirical study of its performance 
%relative to other well-known bounds, 
and provide a method for computing it efficiently. 

%The family of bounds presented below were derived by modifying Gaffke's bound in a way that led to a proof of guaranteed coverage. Gaffke's bound is tighter than our own from our empirical studies, but still eludes a proof. 

We demonstrate in Section~\ref{sec:comparison} that our bound dominates those of both Hoeffding and Anderson. \textbf{To our knowledge, this is the first bound that has been shown to dominate Anderson's bound.}

%{\bf Other bounds.} Many bounds, such as Bennett's \citep{bennett1962probability} and Median of Means, are expressed as a function of the variance of the unknown distribution. These bounds can be used without knowledge of the variance by first bounding the variance, and then using this estimate within the bound together with an application of the union bound, but generally these bounds are not competitive with Anderson's for small sample sizes.

\section{A Family of Confidence Bounds}
%\my{the texts copied from the arxiv report are gray}
In this section we define our new upper confidence bound.  Let $n$ be the sample size. We use bold-faced letters to denote a vector of size $n$ and normal letters to denote a scalar. Uppercase letters denote random variables and lowercase letters denote values taken by them. For example, $X_i \in \mathcal{R}$ and $\X = (X_1, ..., X_n) \in \mathcal{R}^{n}$ are random variables. $x_i \in \mathcal{R}$ is a value of $X_i$, and $\x = (x_1, ..., x_n) \in \mathcal{R}^{n}$ is a value of $\X$. 
For a sample $\x$, we let $F(\x) \definedas (F(x_1), \cdots, F(x_n)) \in [0,1]^n$.

Order statistics play a central role in our work. We denote random variable order statistics $X_{(1)}\leq X_{(2)}\leq ...\leq X_{(n)}$ and of a specific sample as $x_{(1)}\leq x_{(2)}\leq ...\leq x_{(n)}.$

Given a sample $\X = \x$ of size $n$ and a confidence level $1-\alpha$, we would like to calculate a UCB for the mean. Let $F$ be the CDF of $X_i$, i.e., the true distribution and $D\subset \mathcal R$ be the support of $F$.  We assume that $D$ has a finite upper bound. Given $D$ and  any function $T: D^n \rightarrow \mathcal{R}$ we will calculate an upper confidence bound $b^{\alpha}_{D,T}(\x)$ for the mean of $F$. 

We show in Lemma~\ref{lem:superset} that if ${D^+}$ is a superset of $D$ with finite upper bound, then $b^{\alpha}_{{D^+},T}(\x) \ge b^{\alpha}_{D,T}(\x)$. Therefore we only need to know a superset of the support with finite upper bound to obtain a guaranteed bound. 
% Let $n$ be the sample size, $x_1,\dotsc,x_n$ be a specific sample (as opposed to the random variables $X_1,\dotsc,X_n$), %. 
% %We assume all samples 
% let $Z_1,\dotsc,Z_n$ be the order statistics of $X_1,\dotsc,X_n$, i.e., the samples sorted in non-decreasing order, and let $\z=z_1,\dotsc,z_n$ be the order statistics of a specific sample.  %$\mathbf{z}=(z_1,z_2,...,z_n)$ to to order statistics of $be sorted ($z_1 \le z_2 \le \cdots \le z_n$). 
% \my{ in the proof we also use $\X, \x, \y, \Y$ to refer to sorted samples. So we need to say we assume all samples are sorted. I think this paragraph should be removed.}
%For any set $C \subseteq \mathcal{R}$, let $s_C \definedas \sup\{x: x \in C\}$. 

Let $s_D \definedas \sup\{x: x \in D\}$.
 We next describe a method for pairing the sample $\x$ with another vector $\l \in [0,1]^n$ to produce a stairstep CDF function $G_{\x,\l}$. Let $x_{(n+1)}\definedas s_D$. Consider the step function $G_{\x, \l}: \mathcal R %(-\infty,s_D]
 \rightarrow [0,1]$ defined from $\l$ and $\x$ as follows (see Figure~\ref{fig:cdf1}): 
\begin{align}
    G_{\x, \l}(x) = \begin{cases}
      0, & \text{if }  x <x_{(1)} \\
      \ell_{(i)}, & \text{if } x_{(i)} \le x < x_{(i+1)}\\
      1, & \text{if } x\geq s_D.
    \end{cases}
\end{align}
In particular, when $\l = (1/n, \dotsc, n/n)$, $G_{\x, \l}$ becomes the empirical CDF. Also note that when $\l = F(\x)$, $\forall x, G_{\x, \l}(x) \le F(x)$, as illustrated in Figure~\ref{fig:cdf2}. 
\begin{figure}[ht!]
\begin{subfigure}{0.45\textwidth}
		\includegraphics[width = \textwidth]{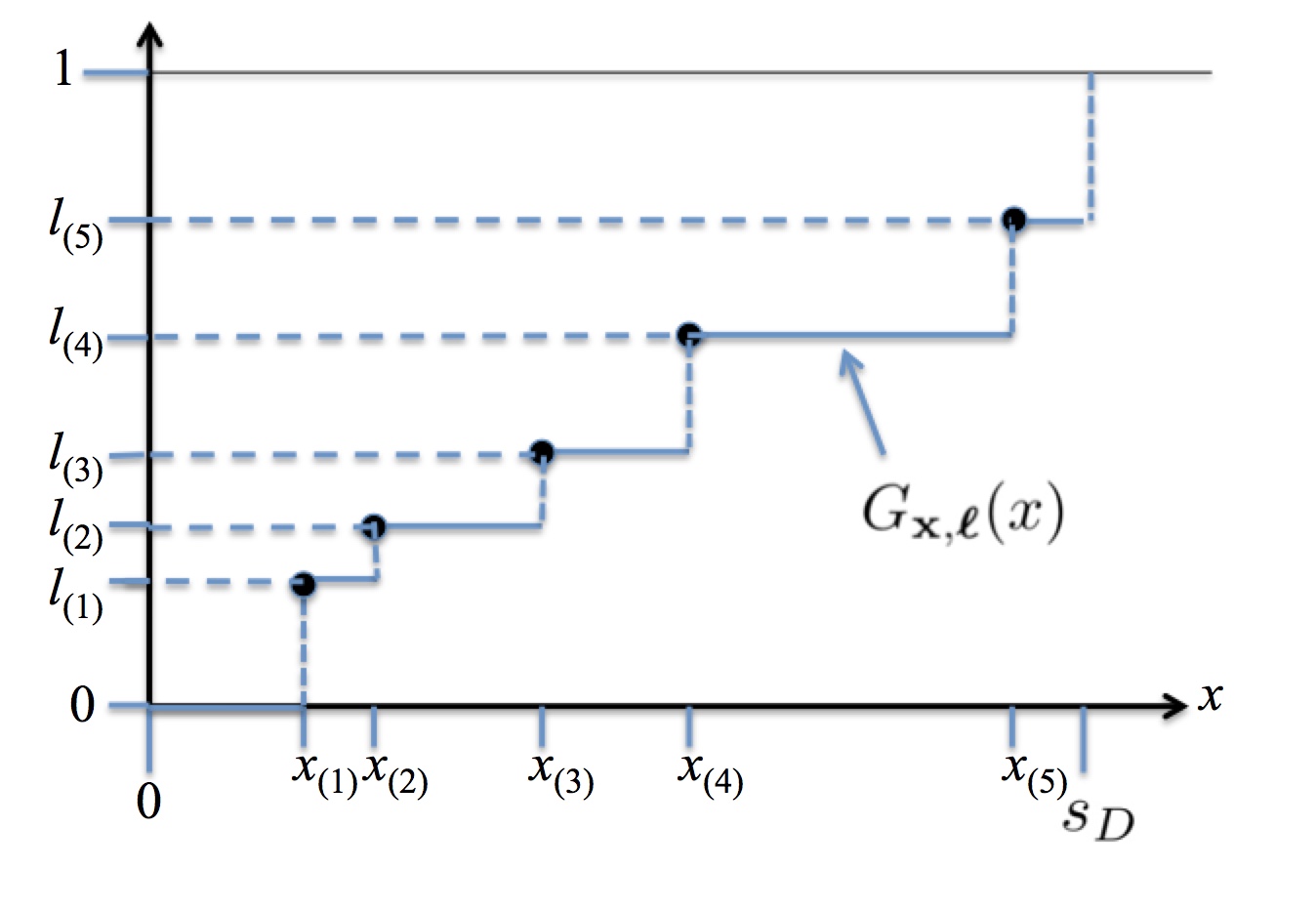}
		\vspace{-.25in}\caption{\label{fig:cdf1} 
		The  stairstep function $G_{\x,\l}$, which is a function of the sample $\x$ and a vector $\l$ of values between 0 and 1. When $\l = (1/n, \dotsc, n/n)$, $G_{\x, \l}$ becomes the empirical CDF.
		}
\end{subfigure}
\quad
\begin{subfigure}{0.45\textwidth}
		\includegraphics[width = \textwidth]{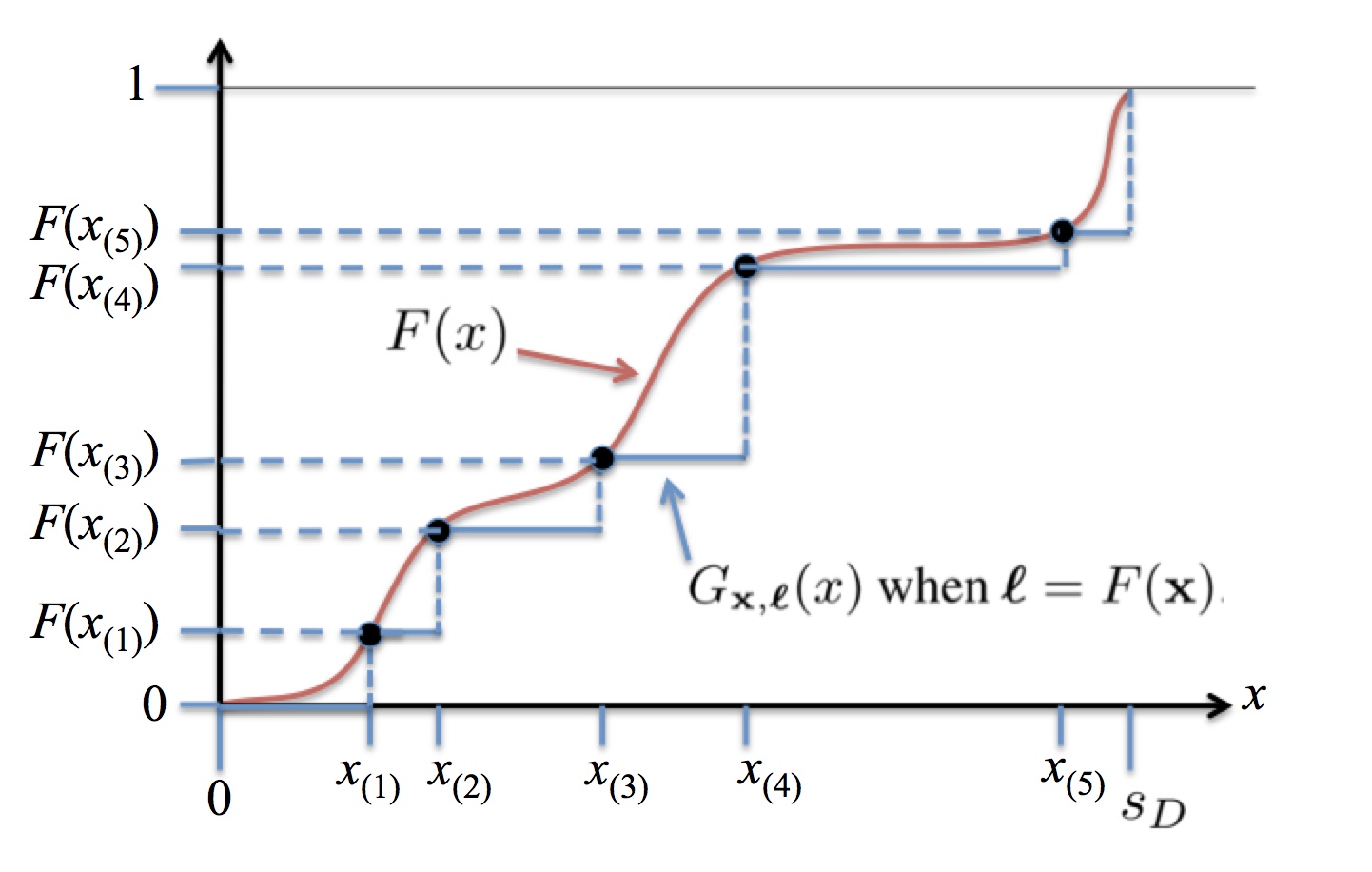}
			\vspace{-.2in}\caption{\label{fig:cdf2} 
		The  CDF of a distribution $F$ in red, with a random sample of five order statistics on the x-axis.  The blue stairstep function shows the 
		function $G_{\x, \l}(x)$ when $\l = F(\x)$. Notice that for all $x$, $G_{\x, \l}(x) \le F(x)$. }
\end{subfigure}
\caption{}
\end{figure}

Following \citet{Learned-MillerThomas2019}, if we consider  $G_{\x, \l}$ to be a CDF, we can compute the mean of the resulting distribution as a function of two  vectors $\x$ and $\l$ as 
% \copied{
%
    %  $$m_D(\x, \l) \definedas \sum_{i=1}^{n+1} x_{(i)}(\ell_{(i)} - \ell_{(i-1)}),$$
\begin{align}
    m_D(\x, \l) &\definedas \sum_{i=1}^{n+1} x_{(i)}(\ell_{(i)} - \ell_{(i-1)}) \\
    &= s_D - \sum_{i=1}^n \ell_{(i)}(x_{(i+1)}- x_{(i)}), \label{eq:induced_mean2} 
\end{align}
% \textcolor{red}{[PST: See higher-than-average priority comment in overleaf here.]}
where $\ell_{(0)} \definedas 0$, $\ell_{(n+1)} \definedas 1$ and $x_{(n+1)} \definedas s_D$. When $s_D$ is finite, this is well-defined. Notice that this function is defined in terms of the {\em order statistics} of $\x$ and $\l$. \citet{Learned-MillerThomas2019} refer to this as the {\em induced mean} for the sample $\x$ by the vector $\l$.  
%}
Although we borrow the above terms from \citet{Learned-MillerThomas2019}, the bound we introduce below is a new class of bounds, and differs from the bounds discussed in their work.

{\bf An ordering on $D^n$.} Next, we introduce a scalar-valued function $T$ which we will use to define a total order on samples in $D^n$, and define a set of samples less than or equal to another sample. 
In particular, for any  function $T:\mathcal{R}^n \rightarrow \mathcal{R}$, let $\mathbb{S}_{D,T}(\x) = \{\y \in D^n | T(\y) \le T(\x) \}$. %(This property implies that $T$ is a total order.) 

{\bf The greatest induced mean for a given $\U$.} Let $\U=U_1,...,U_n$ be a sample of size $n$ from the continuous uniform distribution on $[0,1]$,
with $\u \definedas (u_1,\cdots,u_n)$ being a particular sample of $\U$. 
%}
%\copied{

Now consider the random quantity 
\begin{align}
\label{eq:b_def}
b_{D,T}(\x,\U) \definedas \sup_{\z\in \mathbb{S}_{D,T}(\x)} m_D(\z,\U), 
\end{align}
which depends upon a fixed sample $\x$ (non-random) and
also on the random variable $\U$. 

%}
%\copied{
{\bf Our upper confidence bound.}
Let $0 < p < 1$. Let $Q(p,Y)$ be the {\em quantile function} of the scalar random variable $Y$, i.e.,
%}
\begin{equation}
    \label{eq:quantileDefn}
    Q(p,Y)\definedas \inf \{y\in \mathbb{R}: F_Y(y)\geq p\},
\end{equation}
where $F_Y(y)$ is the CDF of $Y$.
We define $b^{\alpha}_{D,T}(\x)$ to be the $(1-\alpha)$-quantile of the random quantity
$b_{D,T}(\x,\U)$.
\begin{definition}[Upper confidence bound on the mean]
Given a sample $\x$ and a confidence level $1-\alpha$: 
\begin{equation}
    \label{eq:def_bound}
b^{\alpha}_{D,T}(\x) \definedas Q(1-\alpha,b_{D,T}(\x,\U)),
\end{equation} 
where $b_{D,T}(\x,\U)$ is defined in Eq.~\ref{eq:b_def}.
\end{definition}

To simplify notation, we drop the superscript and subscripts whenever clear. 
We show in Section~\ref{sec:guarantee} that this UCB has guaranteed coverage for all sample sizes $n$, for all confidence levels $0< 1-\alpha < 1$ and for all distributions $F$ and support $D$ where $s_D$ is finite.

We show below that a bound computed from a superset ${D^+} \supseteq D$ will be looser than or equal to a bound computed from the support $D$. Therefore it is enough to know a superset of the support $D$ to obtain a bound with guaranteed coverage.% guaranteed bound. 

\begin{lemma}
\label{lem:superset}
Let ${D^+} \supseteq D$ where $s_{D^+}$ is finite. For any sample $\x$:
\begin{align}
    b^{\alpha}_{D}(\x) \le   b^{\alpha}_{{D^+}}(\x).
\end{align}
\end{lemma}
\begin{proof}
Since $s_{D^+}$ is finite, $m_{{D^+}}(\y, \mathbf{\u})$ is well-defined. Since $D \subseteq {D^+}$, for any $\y$ and $\u$,  $m_D(\y, \u) \le m_{{D^+}}(\y, \u)$. Then
\begin{align}
    \sup_{\y \in \mathbb{S}_D(\x)}  m_D(\y, \mathbf{\u}) &\le \sup_{\y \in \mathbb{S}_D(\x)}  m_{{D^+}}(\y, \mathbf{\u}) \\ &\le \sup_{\y \in \mathbb{S}_{{D^+}}(\x)}  m_{{D^+}}(\y, \mathbf{\u}) ,
\end{align}
where the last inequality is because $\mathbb{S}_D(\x) \subseteq \mathbb{S}_{{D^+}}(\x)$. 
Let $b_D(\x, \U)  = \sup_{\z \in \mathbb{S}_D(\x)}  m_D(\z, \mathbf{U})$ and $ b_{{D^+}}(\x, \U)  = \sup_{\z \in \mathbb{S}_{{D^+}}(\x)}  m_{{D^+}}(\z, \mathbf{U}) $. Then $b^{\alpha}_{D}(\x)$ and $b^{\alpha}_{{D^+}}(\x)$ are the $(1-\alpha)$-quantiles of $b_{D}(\x, \U)$ and $b_{{D^+}}(\x, \U)$. Since $b_{D}(\x, \u) \le b_{{D^+}}(\x, \u)$ for any $\u$, $b_{D}^{\alpha}(\x) \le b^{\alpha}_{{D^+}}(\x)$. 
\end{proof}

In Section~\ref{sec:guarantee} we show that the bound has guaranteed coverage. In Section~\ref{sec:computation} we discuss how to efficiently compute the bound. In Section~\ref{sec:comparison} we show that when $T$ is a certain linear function, the bound is equal to or tighter than Anderson's for any sample. In addition, we show that when the support is known to be $\{0,1\}$, our bound recovers the well-known Clopper-Pearson confidence bound for binomial distributions~\citep{Clopper1934}. In Section \ref{sec:simulation}, we present simulations that show the consistent superiority of our bounds over previous bounds.

% \begin{algorithm}[hb]
%   \caption{Find Vertices}
%   \label{alg:vertices}
% \begin{algorithmic}
%   \STATE {\bfseries Input:} dimension $n$. 
%   \STATE {\bfseries Output:} The vertices of the $n$-dimensional polyhedron defined by $0 \le x_1 \le \cdots \le x_n \le 1$.  
%   \STATE Initialize $vertices(1) = [[0],[1]]$.
%   \FOR{$i=2$ {\bfseries to} $n$}
%   \STATE Initialize $vertices(i)$ to be an empty list. 
%   \FOR{ each vertex $v$ in $vertices(i-1)$}
%   \STATE Create $v'$ by add $1$ to be the $i$-th coordinate of $v$
%   \STATE Add $v'$ to $vertices(i)$
%   \ENDFOR
%   \STATE Add vertex $v = [0, \cdots, 0] \in R^n$ to $vertices(n)$.
%   \ENDFOR
%   \STATE Return $vertices(n)$
% \end{algorithmic}
% \end{algorithm}

\subsection{Guaranteed Coverage} 
\label{sec:guarantee}
In this section we show that our bound has guaranteed coverage in Theorem~\ref{thm:guarantee}. We omit superscripts and subscripts if they are clear from context. 

\subsubsection{Preview of proof}
\label{sec:intuition}
\begin{figure*}[ht!]
\begin{center}
		\includegraphics[width = 0.8\textwidth]{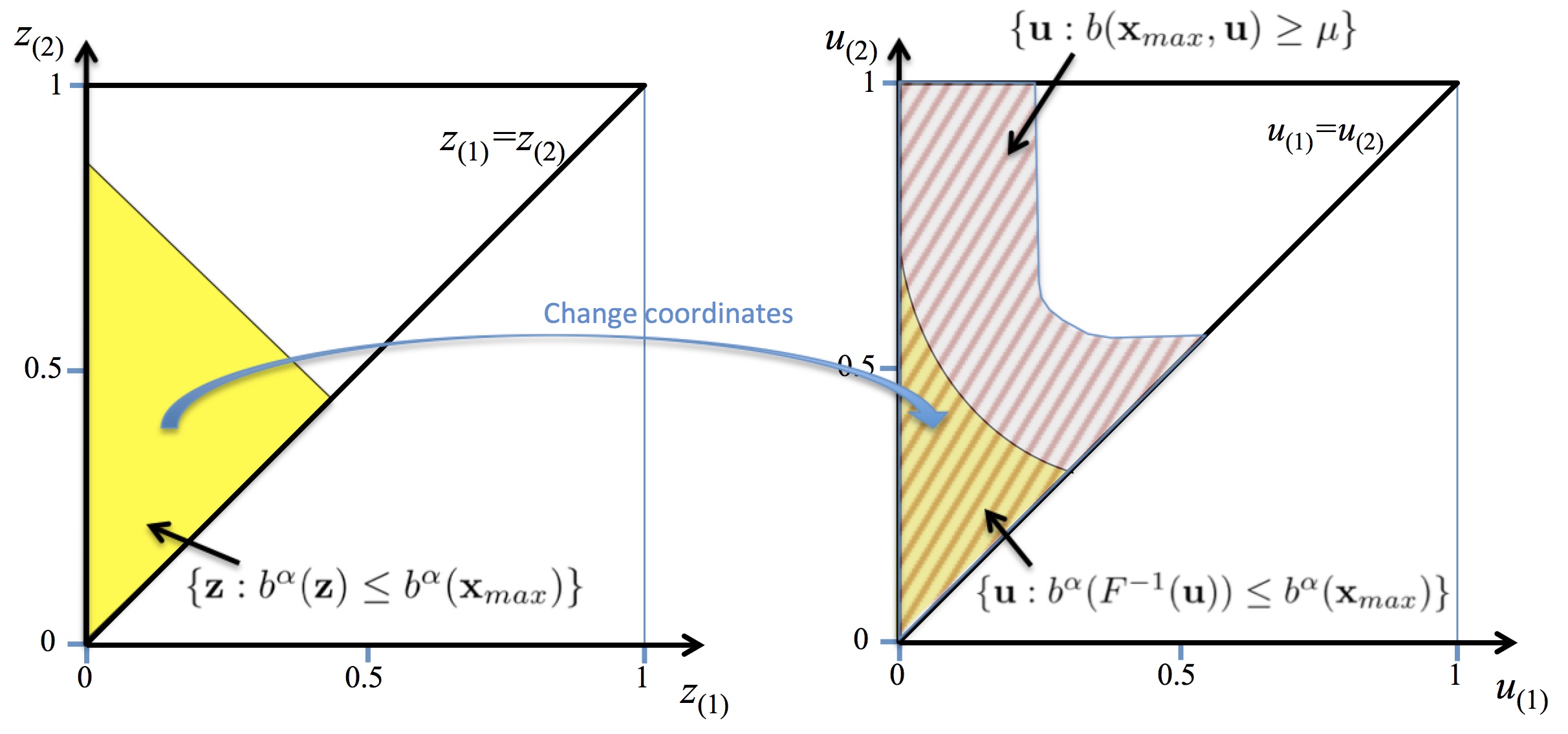}
		\caption{\label{ fig:intuition} Illustrations of Section~\ref{sec:intuition}. {\bf Left.} The yellow region shows samples of $\z=[z_{(1)},z_{(2)}]$ such that $b^\alpha(\z) \leq b^\alpha(\x_{max})$. {\bf Right.} The same yellow region, but in the coordinates  $\u=F^{-1}(\z)$. We will show that the yellow region is a subset of the striped, which contains $\u$ such that $b(\x_{max},\u)\geq \mu$.
		 }
		 \end{center}
\end{figure*}
We explain the idea behind our bound at a high level using a special case. Note that our proof is more general than our special case, which makes assumptions such as the continuity of $F$ to simplify the intuition.

%\textcolor{red}{Here, I will try a new version of the intuitino behind the proof. I will copy the previous version and leave it below, so we can easily revert to it.}

Suppose that $F$ is continuous. Then the {\em probability integral transform} $F_X(X)$ of $X$ is uniformly distributed on $[0,1]$ \citep{10.1137/1036146}. Suppose there exists a sample $\x_{\mu}$ such that $b^{\alpha}(\x_\mu) = \mu$. Then the probability that a sample $\Z$ outputs $b^{\alpha}(\Z) < \mu$ is equal to the probability $\Z$ outputs $b^{\alpha}(\Z) < b^{\alpha}(\x_{\mu})$ (the yellow region on the left of Fig.~\ref{ fig:intuition}). This is the region where the bound fails, and we would like to show that the probability of this region is at most $\alpha$.

% \textcolor{red}{Phil: Below should $b^{\alpha}(\x)\}$ be $b^{\alpha}(\x_{max})\}$. Actually, should this be the case for all the the $\x$ below until the main result? }

Let  $\U \definedas F(\Z)$ and $\u \definedas F(\z)$. Then $U_i$ is uniformly distributed on $[0,1]$. If $F$ is invertible, we can transform the region $\{\z: b^{\alpha}(\z) < b^{\alpha}(\x_{\mu})\}$ to $\{\u: b^{\alpha}(F^{-1}(\u)) < b^{\alpha}(\x_{\mu})\}$ where $F^{-1}(\u) \definedas (F^{-1}(u_1), \dotsc, F^{-1}(u_n))$ (the yellow region on the right of Fig.~\ref{ fig:intuition}).

Through some calculations using the definition of function $b$, we can show that the yellow region  $\{\u: b^{\alpha}(F^{-1}(\u)) < b^{\alpha}(\x_{\mu})\}$ is a subset of the striped region $\{ \u: b(\x_{\mu}, \u) \ge \mu\}$. 

Note that since $b^{\alpha}(\x_{\mu}) = \mu$, $\mu$ is equal to the $1-\alpha$ quantile of $b(\x_{\mu},\U)$. Therefore, by the definition of quantile, the probability of the striped region is at most $\alpha$: 
\begin{align}
    \PP_{\U} ( b(\x_{\mu}, \U) \ge \mu )  \le \alpha,
\end{align}
and thus the probability of the yellow region is at most $\alpha$.

\subsubsection{Main Result}
\label{sec:main_result}
In this section, we present some supporting lemmas and then the main result in Theorem~\ref{thm:guarantee}. The proofs of the simpler lemmas have been deferred to the supplementary material.
%\my{ find a reference for this lemma so we wont need to prove}
\begin{lemma}
\label{lem:CDF_vs_U}
Let $X$ be a random variable with CDF $F$ and $Y \definedas F(X)$, known as the probability integral transform of $X$. Let $U$ be a uniform random variable on $[0,1]$. Then for any $0 \le y \le 1$, 
\begin{align}
    \PP(Y \le y) \le \PP(U \le y).
\end{align}
If $F$ is continuous, then $Y$ is uniformly distributed on $[0,1]$. 
\end{lemma}

In the next lemma we show that the mean satisfies the following property. Let $F$ and $G$ be two CDF functions such that $F(x)$ is always larger than or equal to $G(x)$ for all $x$. Then the mean of $F$ is smaller than the mean of $G$.
\begin{lemma}
\label{lem:helper2}
Let $F$ and $G_{\x,\l}$ be two CDF functions such that $\forall x \in \mathcal{R}$, $F(x) \ge G_{\x,\l}(x)$. Let $\mu_F$ and $\mu_G$ denote the means of $F$ and $G_{\x,\l}$. Then\footnote{This is the only property required of the mean for the subsequent lemmas and theorems. Since quantiles of a distribution also satisfy this condition, this method could also be used to give UCBs for various quantiles.}
\begin{align}
  \mu_F \le \mu_G. 
\end{align}
\end{lemma}

For use in the next lemma, we define a partial order for the samples on $D^n$. Note that it is defined with respect to the {\em order statistics} of the sample, not the original components.

\begin{definition}[Partial Order]
For any two samples $\z$ and $\y$, we define $\z \preceq \y $ to indicate that ${z}_{(i)} \leq {y}_{(i)},  1 \le i \le n$. 
\end{definition}

\begin{lemma}
\label{lem:helper}
Let $\Z$ be a random sample of size $n$ from $F$. Let $\U=U_1,...,U_n$ be a sample of size $n$ from the continuous uniform distribution on $[0,1]$. For any function $T: D^n \rightarrow R$ and any $\x \in D^n$:
\begin{align}
    \PP_{\Z} (T(\Z) \le T(\x)) \le \PP_{\U} (b(\x, \mathbf{U}) \ge \mu  ).
\end{align}
\end{lemma}
\begin{proof}[Proof sketch]

Let $\cup$ denote the union of events and $\{\}$ denote an event. Then for any $\x \in D^n$:
% \textcolor{red}{PST: See comment in overleaf.}
\begin{eqnarray}
    \PP_{\Z} (T(\Z) \le T(\x)) %\label{eq:z_space}
    &= &\PP_{\Z}(\Z \in \mathbb{S}(\x)) 
    \\ 
    &= &\PP_{\Z} (\cup_{\y \in \mathbb{S}(\x)} \{\Z = \y \})
    \\ 
    &\le &\PP_{\Z} (\cup_{\y \in \mathbb{S}(\x)} \{\Z \preceq \y\}) 
    \\ 
    &\le &\PP_{\Z} (\cup_{\y \in \mathbb{S}(\x)} \{F(\Z) \preceq F(\y)\})  \text{ by monotone $F$}\;\;\;
    \\
     &\le &\PP_{\U} (\cup_{\y \in \mathbb{S}(\x)} \{\U \preceq F(\y)\}).  
 \label{eq:u_space_bound}
    \end{eqnarray}
    The last step is by an extension of Lemma~\ref{lem:CDF_vs_U}. Recall that $m_D(\y, \u )=  s_D - \sum_{i = 1}^{n} u_{(i)}(y_{(i+1)} - y_{(i)})$ where $\forall i, y_{(i+1)} - y_{(i)} \ge 0$. Therefore if $\u \preceq F(\y)$
     then $m_D(\y, \u ) \ge m_D(\y, F(\y))$: 
    \begin{eqnarray}
    \PP_{\U} (\cup_{\y \in \mathbb{S}(\x)} \{\U \preceq F(\y)\})
     &\le& \PP_{\U} (\cup_{\y \in \mathbb{S}(\x)}  \{ m_D(\y, \mathbf{U}) \ge m_D(\y, F(\y) ) \}), \text{by Lemma~\ref{lem:helper2}} \label{eq:u_space_bound1} \\
    &\le& \PP_{\U} (\cup_{\y \in \mathbb{S}(\x)} \{  m_D(\y, \mathbf{U}) \ge \mu \} ),   \text{by Lemma~\ref{lem:helper2}} \hspace{.2in}\label{eq:u_space_bound2} \\
    &\le& \PP_{\U} (\sup_{\y \in \mathbb{S}(\x)}   m_D(\y, \mathbf{U}) \ge \mu  )  \\
    &=& \PP_{\U} (b(\x, \mathbf{U}) \ge \mu  ). \label{eq:final}
\end{eqnarray}
% The inequality in Eq.~\ref{eq:sup} is because if there exists $\y \in \mathbb{S}(\x)$ such that $m_D(\y, \mathbf{U}) \ge \mu$, then $\sup_{\y \in \mathbb{S}(\x)}   m_D(\y, \mathbf{U}) \ge \mu$. Therefore the event $\cup_{\y \in \mathbb{S}(\x)} \{  m_D(\y, \mathbf{U}) \ge \mu \}$ is a subset of the event $\sup_{\y \in \mathbb{S}(\x)}   m_D(\y, \mathbf{U}) \ge \mu$, and Eq.~\ref{eq:sup} follows.
\end{proof}
We include a more detailed version of the proof for the above lemma in the supplementary material.

%Using Lemma~\ref{lem:helper} we now show that:
\begin{lemma}
Let $\U=U_1,...,U_n$ be  a sample of size $n$ from the continuous uniform distribution on $[0,1]$. Let $\X$ and $\Z$ denote i.i.d. samples of size $n$ from $F$. For any function $T: D^n \rightarrow \mathcal{R}$ and any $\alpha \in (0,1)$,  
\label{lem:helper4}
    \begin{align}
 \PP_{\X} \left(\PP_{\U} (b_{D,T}(\X, \U) \ge \mu) \le \alpha \right) \notag 
    &\le \PP_{\X} \left(\PP_{\Z} (T(\Z) \le T(\X)) \le  \alpha \right) .
\end{align}
\end{lemma}
\begin{proof}
From Lemma~\ref{lem:helper} for any sample $\x$,
\begin{align}
    \PP_{\Z} (T(\Z) \le T(\x)) \le \PP_{\U} (b(\x, \mathbf{U}) \ge \mu  ) .
\end{align}
% For any two random variables $W$ and $V$, we note that if $W \ge V$ then
% \begin{align}
% \label{eq:helper}
%     P( W \le \alpha ) \le P(V \le \alpha). 
% \end{align}

Therefore, % using Eq.~\ref{eq:helper} with $W = \PP_{\U} (b(\X, \mathbf{U}) \ge \mu )$ and $V = \PP_{\Z} (T(\Z) \le T(\X))$: 
\begin{align}
     \PP_{\X} \left(\PP_{\Z} (T(\Z) \le T(\X)) \le  \alpha \right) 
    &\ge \PP_{\X} \left(\PP_{\U} (b(\X, \U) \ge \mu) \le \alpha \right) .
\end{align}
%\vspace{-.25in}
\end{proof}
\begin{lemma}
\label{lem:helper3}
Let $\U=U_1,...,U_n$ be a sample of size $n$ from the continuous uniform distribution on $[0,1]$. Let $\X$ be a random sample of size $n$ from $F$. For any function $T: D^n \rightarrow \mathcal{R}$ and any $\alpha \in (0,1)$,  
\begin{align}
    \PP_{\X} (b_{D,T}^{\alpha}(\X) < \mu)  &\le \PP_{\X} \left(\PP_{\U} (b_{D,T} (\X, \U) \ge \mu) \le  \alpha \right). 
\end{align}
\end{lemma}
\begin{proof}
Because $b^{\alpha}(\x)$ is the $1-\alpha$ quantile of $b(\x, \mathbf{U})$,  by the definition of quantile: $\PP_{\U} (b(\x, \U) \le b^{\alpha}(\x))  \ge 1- \alpha$. 
Therefore
$\PP_{\U} (b(\x, \U) \ge b^{\alpha}(\x))  \le \alpha.$ If $b^{\alpha}(\x) < \mu$ then
$\PP_{\U} (b(\x, \U) \ge \mu) \le \alpha.$ 
Since $b^{\alpha}(\x) < \mu$  implies $\PP_{\U} (b(\x, \U) \ge \mu) \le \alpha$, we have
\begin{align}
   \PP_{\X} (b^{\alpha}(\X) < \mu)  &\le \PP_{\X} \left(\PP_{\U} (b(\X, \U) \ge \mu) \le  \alpha \right).  \label{eq:helper3}
\end{align}
%\vspace{-.1in}
\end{proof}
% Set $V = \PP_{\U} (b(\X, \U) \ge \mu)$. We want to show that: 
% \begin{align}
%     P_V(V \le \alpha) \le \alpha
% \end{align}
% We do this by finding another R.V W where we know that: 
% \begin{align}
%     P_W(W \le \alpha) \le \alpha
% \end{align}
% and the cdf of W is always larger than V. Then 
% \begin{align}
%      P_V(V \le \alpha)  \le P_W(W \le \alpha) \le \alpha
% \end{align}
% We found $W$ to be the one below. 
\vspace{-.15in}
We now show that the bound has guaranteed coverage. 
\begin{theorem}
\label{thm:guarantee}
Let $\X$ be a random sample of size $n$ from $F$. For any function $T: D^n \rightarrow R$ and for any $\alpha \in (0,1)$: 
\begin{align}
  \PP_{\X}(b^\alpha_{D,T}(\X) < \mu) \le \alpha. 
    \end{align}
\end{theorem}
\begin{proof}
Let $\Z$ be a random sample of size $n$ from $F$.
\begin{align}
    \PP_{\X} (b^{\alpha}(\X) < \mu)  &\le \PP_{\X} \left(\PP_{\U} (b(\X, \U) \ge \mu) \le  \alpha \right) \text{ by Lemma~\ref{lem:helper3}}\\
    &\le \PP_{\X} \left(\PP_{\Z} (T(\Z) \le T(\X)) \le  \alpha \right) \text{ by Lemma~\ref{lem:helper4}} \\
    &= \PP \left(W \le  \alpha \right) \text{ where } W \definedas \PP_{\Z} (T(\Z) \le T(\X))\\
    &\le \alpha \text{ by Lemma~\ref{lem:CDF_vs_U}}. 
\end{align}
\vspace{-.2in}
\end{proof}
% Now we analyze $\PP_{\Z} (b^{\alpha}(\Z) \le \mu )$:
% because $b^{\alpha}(\z)$ is the $1-\alpha$ quantile of $b(\z, \mathbf{U})$,  by the definition of quantile: $\PP_{\U} (b(\z, \U) \le b^{\alpha}(\z))  \ge 1- \alpha$. 

% Therefore:
% \begin{align}
% \PP_{\U} (b(\z, \U) \ge b^{\alpha}(\z))  \le \alpha.
% \end{align}

% If $b^{\alpha}(\z) < \mu$ then:
% $$\PP_{\U} (b(\Z, \U) \ge \mu) \le \alpha$$. 

% Since $b^{\alpha}(\z) < \mu$  implies $\PP_{\U} (b(\Z, \U) \ge \mu) \le \alpha$, we have:
% \begin{align}
%   &\PP_{\Z} (b^{\alpha}(\Z) < \mu) \\ &\le \PP_{\Z} \left(\PP_{\U} (b(\Z, \U) \ge \mu) \le  \alpha \right) \label{eq:helper3}
% \end{align}

\section{Computation}
\label{sec:computation}
In this section we present a Monte Carlo algorithm to compute the bound. First we note that since the bound only depends on $\x$ via the function $T(\x)$, we can precompute a table of the bounds for each value of $T(\x)$. We discuss how to adjust for the uncertainty in the Monte Carlo result in Appendix~\ref{apx:mc}.
\begin{algorithm}[!ht]
 \caption{ \label{alg:monte_carlo} Monte Carlo estimation of $m^\alpha_{D^+,T}(\x)$ where $D^+= [0,1]$. This pseudocode uses 1-based array indexing.}
 \begin{algorithmic}
\STATE {\bfseries Input:} {A sample $\x \in D^n$, confidence parameter $1-\alpha < 1$, a function $T: [0,1]^n \rightarrow \mathcal{R}$ and Monte Carlo sampling parameter $l$.}
\STATE {\bfseries Output:} {An estimation of $m^\alpha_{D^+,T}(\x)$} \\
$n \leftarrow length(\x)$. \\
\STATE Create array $\textbf{ms}$ to hold $l$ floating point numbers, and initialize it to zero. \\
\STATE Create array $\textbf{u}$ to hold $n$ floating point numbers. \\
%\STATE Sort(\textbf{x}, ascending).\\
 \FOR {$i \leftarrow 1$ to $l$}{
  \FOR {$j \leftarrow 1$ to $n$}{
  \STATE \textbf{u}[j] $\sim$ Uniform(0,1).
  }
  \ENDFOR
  \STATE Sort(\textbf{u}, ascending). \\
 \STATE  Solve: $M = \max_{y_{(1)}, \cdots, y_{(n)}}  m(\y,\u)$ subject to: \\
 \begin{ALC@g}
    \STATE   1) $T(\y) \le T(\x)$.\\
     \STATE 2) $~\forall i: 1\le i \le n, 0 \le y_{(i)} \le 1$.\\
     \STATE   3) $y_{(1)} \le y_{(2)} \le ... \le y_{(n)}$.\\
     \end{ALC@g}
   \STATE $\mathbf{ms}[i] = M$.
 }
 \ENDFOR
 \STATE Sort(\textbf{ms}, ascending). \\
 \STATE Return $\mathbf{ms}[\lceil (1-\alpha)l\rceil]$.
 \end{algorithmic}
\end{algorithm}

Let the superset of the support $D^+$ be a closed interval with a finite upper bound. If $m$ is a continuous function, 
\begin{align}
\label{eq:linProg}
    \sup_{\y \in \mathbf{S}_{D^+}(\x)} m(\y,\u) = \max_{\y \in \mathbf{S}_{D^+}(\x)} m(\y,\u) .
\end{align}
Therefore $b_{D^+}(\x,\u)$ is the solution to
\begin{align}
&\max_{y_{(1)}, \dotsc, y_{(n)}} m(\y,\u) 
% \\ &=  \max_{y_{(1)}, \cdots, y_{(n)}} &&y_{(1)} u_{(1)} + \sum_{i=2}^{n} y_{(i)} (u_{(i)} - u_{(i-1)}) \\ & &&+ s_{D^+}(1 - u_{(n)})
\end{align}
subject to: 
\begin{compactenum}
   \item  $T(\y) \le T(\x)$,
     \item $~\forall i \in \{1,\dotsc,n\}, %: 1\le i \le n,
     y_{(i)} \in {D^+}$,
     \item $y_{(1)} \le y_{(2)} \le \cdots \le y_{(n)}$. 
     \end{compactenum}
When $D^+$ is an interval and $T$ is linear, this is a linear programming problem and can be solved efficiently. 

We can compute the $1-\alpha$ quantile  of a random variable $M$ using Monte Carlo simulation, sampling $M$ $l$ times. Letting $m_{(1)}\le ... \le m_{(l)}$ be the sorted values, we output $m_{(\lceil (1-\alpha)l\rceil)}$ as an approximation of the $1-\alpha$ quantile.

{\bf Running time.} Note that since the bound only depends on $\x$ via the function $T(\x)$, we can precompute a table of the bounds for each value of $T(\x)$ to save time.  When $T$ is linear, the algorithm needs to solve a linear programming problem with $n$ variables and $2n$ constraints $l$ times. 
For sample size $n = 50$, computing the bound for each sample $\x \in D^n$ takes just a few seconds using $l = 10,\!000$ Monte Carlo samples.

\section{Relationships with Existing Bounds}
\label{sec:comparison}
In this section, we compare our bound to previous bounds including those of Clopper and Pearson, Hoeffding, and Anderson.
Proofs omitted in this section can be found in the supplementary material.

\subsection{Special Case: Bernoulli Distribution}
\label{sec:Bernoulli}
When we know that $D = \{0,1\}$, the distribution is Bernoulli. If we choose $T$ to be the sample mean, our bound becomes the same as the Clopper-Pearson confidence bound for binomial distributions~\citep{Clopper1934}. See the supplementary material for details.

\subsection{Comparisons with Anderson and Hoeffding}
\label{sec:vsAnderson}
%\my{ this paragraph needs to be rewritten to be less confusing}
In this section we show that for any sample size $n$, any confidence level $\alpha$ and for any sample $\x$,  our method produces a bound no larger than Anderson's bound (Theorem~\ref{thm:vsAnderson}) and Hoeffding's bound (Theorem~\ref{thm:vsHoeffding}). 

Note that if we only know an upper bound $b$ of the support (1-ended support setting), we can set ${D^+} = (-\infty, b]$ %. Even when we have no bound on the lower support, or when the lower bound is known to be $-\infty$ 
and our method is equal to Anderson's (Appendix~\ref{apx:equal_Anderson})  and dominates Hoeffding's. As the lower support bound increases (2-ended setting), our bound becomes tighter or remains constant, whereas Anderson's remains constant, as it does not incorporate information about a lower support. Thus, in cases where our bound can benefit from a lower support, we are tighter than Anderson's.
%Let $a$ and $b$ be the known lower bound and upper bound of the support. 

Anderson's bound constructs an upper bound for the mean by constructing a lower bound for the CDF. We defined a lower bound for the CDF as follows. 
\begin{definition}[Lower confidence bound for the CDF]
\label{cond:step-wise-cdf}
Let $\X = (X_1,  \cdots, X_n)$ be a sample of size $n$ from the distribution on $\mathcal{D}^+$ with unknown CDF $F$. Let $\alpha \in (0,1)$. Let $H_{\X}: \mathcal{R} \rightarrow [0,1]$ be a function computed from the sample $\X$ such that for any CDF $F$, 
\begin{align}
\label{eq:lowerCDFbound}
    \PP_{\X} (~\forall x \in R, F(x) \ge H_{\X}(x)) \ge 1 - \alpha.
\end{align}
Then $H_{\X}$ is called a $(1-\alpha)$ lower confidence bound for the CDF. 

If there exists a CDF $F$ such that
\begin{align}
    \PP_{\X} (~\forall x \in R, F(x) \ge H_{\X}(x)) = 1 - \alpha,
\end{align}
then $H_{\X}$ is called an \emph{exact} $(1-\alpha)$ lower confidence bound for the CDF. 
\end{definition}

%\textcolor{red}{This next part is confusing. Are you saying our bound is better than Anderson's even when Anderson uses $[a,b]$ and we only use $[-\infty,b]$ as supports? Needs to be clarified.}

% We now show that if $G_{\X, \l}$ (Figure~\ref{fig:cdf1}) is a lower confidence bound, then the order statistics of $\l$ are element-wise smaller than the order statistics of a sample of size $n$ from the uniform distribution with high probability: 
% \begin{lemma}
% \label{lem:lower_bound_cdf}
% Let $\U=U_1,...,U_n$ be a sample of size $n$ from the continuous uniform distribution on $[0,1]$. Let $\l \in [0,1]^n$ and $\alpha \in (0,1)$. If $G_{\X, \l}$ is a $(1-\alpha)$ lower confidence bound for the CDF then:
% \begin{align}
% \label{eq:vector_bound}
%     \PP_{\U} (~\forall i: 1 \le i \le n, U_{(i)} \ge \ell_{(i)}) \ge 1 - \alpha.
% \end{align}
% \end{lemma}

In Figs.~\ref{fig:cdf1} and \ref{fig:cdf2}, it is easy to see that if the stairstep function $G_{\X, \l}$ is a lower confidence bound for the CDF then its induced mean $m(\X, \l)$ is an upper confidence bound for $\mu$.
\begin{lemma}
\label{lem:induced_mean_bound}
Let $\X = (X_1,  \cdots, X_n)$ be a sample of size $n$ from a distribution with mean $\mu$. Let $\l \in [0,1]^n$. If $G_{\X, \l}$ is a $(1-\alpha)$ lower confidence bound for the CDF then % $m(\X, \l)$ is a  upper confidence bound for $\mu$:
\begin{align}
    \PP_{\X}(m(\X, \l) \ge \mu) \ge 1 - \alpha. 
\end{align}
\end{lemma}

%\textcolor{red}{The statement that Anderson uses DKW is actually not correct. DKW is not a tight boudn on the CDF in the sense that it gives a probability of slightly more than $1-\alpha$ of containing the CDF. Anderson's bound is actually defined in terms of unspecified values such taht the envelope is tight. I will replace this with the right stuff... }
Let $U_{(i)}, 1 \le i \le n$ be the order statistics of the uniform distribution. Note that for any CDF $F$: 
\begin{align}
    \PP_{\X} (~\forall x \in \mathcal{R}, F(x) \ge G_{\X, \l}(x)) 
      & = \PP_{\X} (\forall i: 1 \le i \le n, F(X_{(i)}) \ge \ell_{(i)}) \\
       &\ge \PP_{\U} (\forall i: 1 \le i \le n, U_{(i)} \ge \ell_{(i)}) \text{ by Lemma~\ref{lem:CDF_vs_U}}, \label{eq:CDF_equal_U}
\end{align}
where Eq.~\ref{eq:CDF_equal_U} is an equality if $F$ is the CDF of a continuous random variable. Therefore  $G_{\X, \l}$ is an exact $(1-\alpha)$ lower confidence bound for the CDF is equivalent to $\l$ satisfying: 
\begin{align}
    \PP_{\U} (\forall i: 1 \le i \le n, U_{(i)} \ge \ell_{(i)}) = 1 - \alpha.
\end{align}

\citet{Anderson1969} presents $b^{\alpha, \text{Anderson}}_{\l} (\x) = m_{D^+}(\x, \l)$ as a UCB for $\mu$ where $\l \in [0,1]^n$ is a vector such that $G_{\X, \l}$ is an exact $(1-\alpha)$ lower confidence bound for the CDF. 

In one instance of Anderson's bound,  $\l = \u^{And} \in [0,1]^n$ is defined as 
\begin{equation}
    \label{eq:And}
     u^\text{And}_i\definedas \max \left \{0,i/n-\beta(n) \right \}.
\end{equation}

Anderson identifies $\beta(n)$ as the one-sided Kolmogorov-Smirnov statistic such that $G_{\X, \l}$ is an exact $(1-\alpha)$ lower confidence bound for the CDF when $\l = \u^\text{And}$.  $\beta(n)$ can be computed by Monte Carlo simulation (Appendix~\ref{apx:exp}).
% Using the fact that $G_{\X, \u^{And}}$ is a $1-\alpha$ lower confidence bound for the CDF, and letting $D^+ = (-\infty, b]$, Anderson's $1-\alpha$ upper bound for the mean is
% \begin{equation}
%     b^{\alpha, \text{Anderson}} (\x)\definedas
%     m_{D^+}(\x, \u^\text{And}).
% \end{equation}

\citet{Learned-MillerThomas2019} show that for any sample $\x$, a looser version of Anderson's bound is better than Hoeffding's:

\begin{lemma}[from Theorem $2$ from \citep{Learned-MillerThomas2019}]
\label{lem:AndersonvsHoeffding}
%Let $a$ and $b$ be the lower and upper bounds of the support. 
For any sample size $n$, for any sample value $\x \in D^n$,  for all $\alpha \in (0,0.5]$:
\begin{align}
\label{eq:AndvsH}
   b^{\alpha, \text{Anderson}}_{\l} (\x) \le   b^{\alpha, \text{Hoeffding}} (\x),
\end{align}
where $\l$ is defined\footnote{Although Anderson's bound $b^{\alpha, \text{Anderson}}_{\l} (\x)$ is only defined when $G_{\X, \l}$ is an exact $(1-\alpha)$ lower confidence bound for the CDF, here we re-use the same notation for the case when $G_{\X, \l}$ is a $(1-\alpha)$ lower confidence bound for the CDF.} as 
\begin{align}
    \ell_i\definedas \max \left \{0,i/n- \sqrt{\ln(1/\alpha)/(2n)} \right \}.
\end{align}
When $\alpha \le 0.5$, this definition of $\l$ satisfies $G_{\X, \l}$ is a $(1-\alpha)$ lower confidence bound for the CDF. 

The inequality in Eq.~\ref{eq:AndvsH} is strict for $n\ge 3$.
\end{lemma}
%Together with Corollary~\ref{col:vsAnderson}, this shows that for any sample $\x$, our bound is better than Hoeffding's (Theorem~\ref{thm:vsHoeffding}). Note that Hoeffding's requires both an upper support bound and a lower support bound while our bound only requires an upper support bound. 

% In the following theorem,  we compare $b^\alpha_{{D^+},T}$ to $m_{{D^+}}(\x, \l)$ if $G_{\X, \l}$ is a lower confidence bound for the CDF.  This theorem shows that our bound is 
% \begin{theorem}
%     \label{thm:LMTVsAnderson}
%      Let $\l \in [0,1]^n$. 
%     Let $D^+ = [-\infty, b]$. If $G_{\X, \l}$ is a  $1-\alpha$ lower confidence bound for the CDF, then for any sample size $n$, for all sample values $\x \in D^n$ and all $\alpha \in (0,1)$, using $T(\x) = m_{{D^+}}(\x, \l)$ to compute $b^\alpha_{{D^+},T} (\x)$ yields:
%     \begin{equation}
%         b^\alpha_{{D^+},T}(\x) \leq m_{{D^+}}(\x, \l).
%     \end{equation}
% \end{theorem}
 We show below that our bound is always equal to or tighter than Anderson's bound. Appendix~\ref{apx:equal_Anderson} provides a more detailed analysis showing that our bound is equal to Anderson's when the lower bound of the support is too small and can be tighter than Anderson's when the lower bound of the support is large enough.
\begin{theorem}
\label{thm:vsAnderson}
Let $\l \in [0,1]^n$ be a vector satisfying  
$G_{\X, \l}$ is an exact $(1-\alpha)$ lower confidence bound for the CDF.

Let $D^+ = (-\infty, b]$. For any sample size $n$, for any sample value $\x \in D^n$,  for all $\alpha \in (0,1)$, using $T(\x) = b^{\alpha, \text{Anderson}}_{\l} (\x)$  yields
\begin{align}
   b^\alpha_{D^+,T}(\x) \le   b^{\alpha, \text{Anderson}}_{\l} (\x).
\end{align}
\end{theorem}
We explain briefly why this is true. First, from Figure~\ref{fig:cdf2}, we can see that if $G_{\X, \l}$ is a lower confidence bound then $\forall i, F(X_{(i)}) \ge \ell_{(i)}$. Note that $G_{\X, \l}$ must be a lower bound for all unknown CDFs $F$, so we can pick a continuous $F$ where, according to Lemma~\ref{lem:CDF_vs_U}, $U \definedas F(X)$ is uniformly distributed on $[0,1]$. Therefore $\l$ satisfies 
\begin{align}
    \PP_{\U}(\forall i, U_{(i)} \ge \ell_{(i)}) \ge 1 - \alpha, 
\end{align}
where the $U_{(i)}$'s are the order statistics of the uniform distribution. Since $b(\x, \U)$ is defined from linear functions of $\U$ with negative coefficients (Eq.~\ref{eq:induced_mean2}), if $\forall i, U_{(i)} \ge \ell_{(i)}$ then $b(\x, \U) \le b(\x,\l)$. 
Therefore with probability at least $1-\alpha$, $b(\x, \U) \le b(\x,\l)$. So $b(\x,\l)$ is at least the $1-\alpha$ quantile of $b(\x,\U)$, which is the value of our bound. Therefore $b(\x,\l)$ is at least the value of our bound. 

Finally, if $T$ is Anderson's bound, through some calculations we can show that $b_{D^+,T}(\x,\l) = m_{D^+}(\x,\l)$, which is Anderson's bound. The result follows. 

% \begin{proof}
% We have $b^{\alpha, \text{Anderson}} (\x) = m_{D^+}(\x, \u^{And})$ where $\u^{And}$ satisfies $G_{\X, \u^{And}}$ is a $1-\alpha$ lower confidence bound for the CDF. 
% Therefore applying Theorem~\ref{thm:LMTVsAnderson} with $\l = \u^{And}$ yields the result. 
% \end{proof}

\begin{figure*}[ht!]
		\centering
		\includegraphics[trim=0 0 0 0,width= \textwidth]{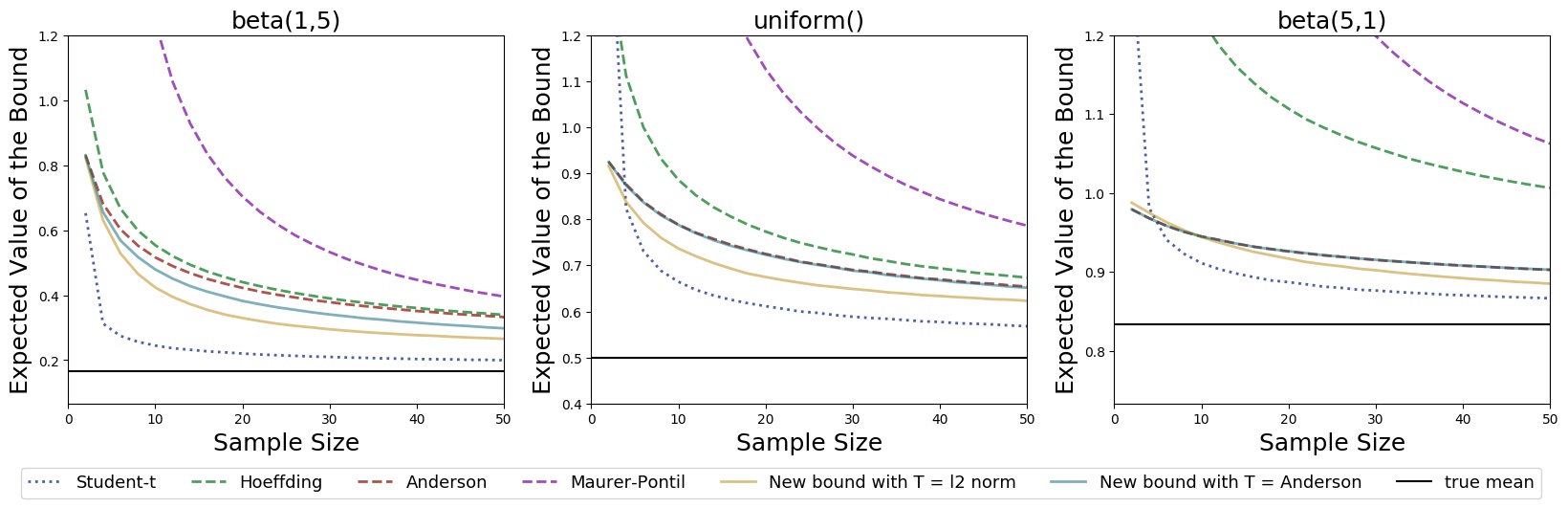}
	\caption{\label{fig:mean}  The expected value of the bounds for $\alpha = 0.05$ and $D^+ = [0,1]$. For each sample size, we sample $\mathbf{X}$ $10,\!000$ times, compute the bound for each sample, and take the average. Our new bound with $T$ being Anderson's bound consistently has lower expected value than Anderson's (Theorem~\ref{thm:vsAnderson}), Hoeffding's (Theorem~\ref{thm:vsHoeffding}) and Maurer and Pontil's. With $T$ being the $l2$-norm, the bound is substantially tighter in these examples, and also has guaranteed coverage.}
	\end{figure*}
	\begin{figure*}[ht!]
		\centering
		\includegraphics[trim=0 0 0 0,width= \textwidth]{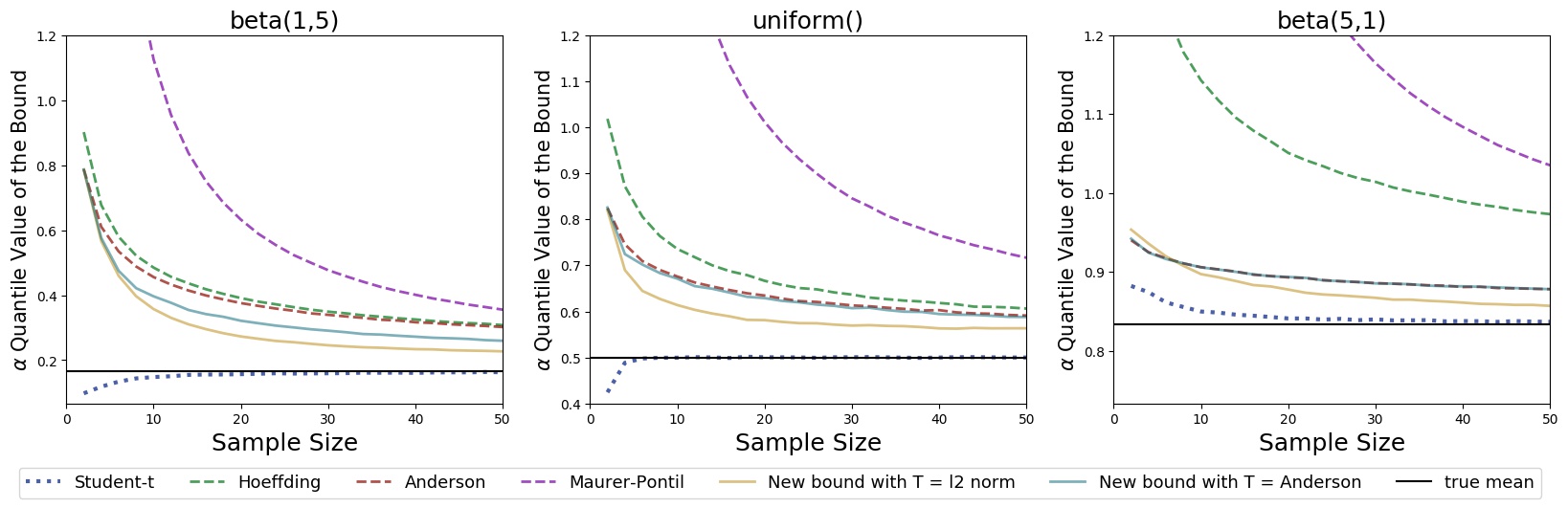}
	\caption{\label{fig:alpha_quantile} The $\alpha$-quantile of the bound distribution for $\alpha = 0.05$ and $D^+ = [0,1]$. For each sample size, we sample $\mathbf{X}$ $10,\!000$ times, compute the bound for each sample, and take the $\alpha$ quantile. If the $\alpha$-quantile is below the true mean, the bound does not have guaranteed coverage. For the ${\tt uniform}(0,1)$ and ${\tt beta}(1,5)$ distribution, when the sample size is small, Student-t does not have guarantee.}
		\end{figure*}

The comparison with Hoeffding's bound follows directly from Lemma~\ref{lem:AndersonvsHoeffding} and Theorem~\ref{thm:vsAnderson}:
\begin{theorem}
\label{thm:vsHoeffding}
Let $D^+ = (-\infty, b]$. For any sample size $n$, for any sample value $\x \in D^n$, for all $\alpha \in (0,0.5]$, using $T(\x) =  b^{\alpha, \text{Anderson}}_{\l} (\x)$  where $\l = \u^{\text{And}}$ yields:
\begin{align}
   b^\alpha_{D^+,T}(\x) \le   b^{\alpha, \text{Hoeffding}} (\x),
\end{align}
where the inequality is strict when $n \ge 3$. 
\end{theorem}
% \begin{proof}
% The proof follows directly from Lemma~\ref{lem:AndersonvsHoeffding} and Theorem~\ref{thm:vsAnderson}
% \end{proof}
\citet{Diouf2005} present several instances of Anderson's bound with different $\l$ computed from the Anderson-Darling or the Eicker statistics (Theorem $4$, $5$ and Theorem $6$ with constant $\epsilon$). 

Note that the result from Theorem~\ref{thm:vsAnderson} can be generalized for bounds $m(\X,\l)$ constructed from a $(1-\alpha)$ confidence lower bound $G_{\X, \l}$ using Lemma~\ref{lem:induced_mean_bound}.  %The result from Theorem~\ref{thm:vsAnderson} can be generalized to these bounds as well. 
%Theorem~\ref{thm:vsAnderson} is also true for other value of $\beta(n)$ used to compute $\u^\text{And}$ as long as $G_{\X, \u^{And}}$ is a $1 - \alpha$ lower CDF confidence bound.
We show the general case in the supplementary material. 
%\clearpage
\section{Simulations}
\label{sec:simulation}
We perform simulations to compare our bounds to 
Hoeffding's inequality, Anderson's bound, Maurer and Pontil's, and Student-t's bound \citep{student1908probable}, the latter being
\begin{equation}
    b^{\alpha,\text{Student}} (\X) \definedas \bar X_n + \sqrt{\frac{\hat \sigma^2}{n}}t_{1-\alpha,n-1}.
\end{equation}
% For Anderson's bound, we use the Dvoretsky-Kiefer-Wolfowitz inequality~\citep{Dvoretzky1956} to define the $1-\alpha$ CDF lower bound via $\beta(n) = \sqrt{\ln(1/\alpha)/(2n)}$. 

We compute Anderson's bound with $\l= \u^{And}$ defined in Eq.~\ref{eq:And} through Monte Carlo simulation (described in Appendix~\ref{apx:exp}). 
We use $\alpha = 0.05$, $D^+ = [0,1]$ and $l = 10,\!000$ Monte Carlo samples. We consider two  functions $T$: 
\begin{enumerate}
    \item Anderson: $T(\x) =  b^{\alpha, \text{Anderson}}_{\l} (\x)$, again with $\l= \u^{And}$. Because this $T$ is linear in $\x$, it can be computed with the linear program in Eq.~\ref{eq:linProg}.
    \item $l2$ norm: $T(\x) = (\sum_{i=1}^n x^2_i)/n$. In this case, $T$ requires the optimization of a linear functional over a convex region, which results in a simple convex optimization problem.
\end{enumerate}
We perform experiments on three distributions: ${\tt beta}(1,5)$ (skewed right), ${\tt uniform}(0,1)$ and ${\tt beta}(5,1)$ (skewed left). Their PDFs are included in the supplementary material for reference. Additional experiments are in the supplementary material.

In Figure~\ref{fig:mean} and Figure~\ref{fig:alpha_quantile} we plot the expected value and the $\alpha$-quantile value of the bounds as the sample size increases. Consistent with Theorem~\ref{thm:vsAnderson}, our bound with $T$ being  Anderson's bound outperforms Anderson's bound. Our new bound performs better than Anderson's in distributions that are skewed right, and becomes similar to Anderson's in left-skewed distributions. Our bound outperforms Hoeffding and Maurer and Pontil's for all three distributions. Student-t fails (the error rate exceeds $\alpha$) for ${\tt beta}(1,5)$ and ${\tt uniform}(0,1)$ when the sample size is small (Figure~\ref{fig:alpha_quantile}). 

\subsection*{Acknowledgements} 

This work was partially supported by DARPA grant FA8750-18-2-0117.  
Research reported in this paper was sponsored in part by NSF award \#2018372 and the DEVCOM Army Research Laboratory under Cooperative Agreement W911NF-17-2-0196 (ARL IoBT CRA). The views and conclusions contained in this document are those of the authors and should not be interpreted as representing the official policies, either expressed or implied, of the Army Research Laboratory or the U.S. Government. The U.S. Government is authorized to reproduce and distribute reprints for Government purposes notwithstanding any copyright notation herein.

\clearpage
\bibliography{mean_interval.bib}
\bibliographystyle{icml2021}
\clearpage
\appendix
\onecolumn
\begin{center}
{\Large Supplementary Material: \\ Towards Practical Mean Bounds for Small Samples}
\end{center}
\begin{itemize}
    \item In Section~\ref{apx:exp} we describe the computation of Anderson's bound and present more experiments. 
    \item In Section~\ref{apx:intro}, as noted in Section~\ref{sec:intro}, we show a log-normal distribution where the sample mean distribution is visibly skewed when $n = 80$.  
    \item In Section~\ref{apx:proof_guarantee}  we present the proofs of Section~\ref{sec:main_result}. 
    \item In Section~\ref{apx:mc} we discuss the Monte Carlo convergence result of our approximation in Section~\ref{sec:computation}. 
    \item In Section~\ref{apx:bernoulli} we show that our bound reduces to the Clopper-Pearson bound for binomial distributions as mentioned in Section~\ref{sec:Bernoulli}. In Section~\ref{apx:proof_vsAnderson} we present the proofs of Section~\ref{sec:vsAnderson}. In Section~\ref{apx:equal_Anderson} we showed that our bound reduces to Anderson's bound when the lower bound of the support is too small. 
\end{itemize}
% \textcolor{red}{Note that the numbers of the lemmas need to be coordinated with the main text.}
\section{Other Experiments}
\label{apx:exp}

In this section we perform experiments to find an upper bound of the mean of distributions given a finite upper bound of the support, or to a lower bound of the mean of distributions given a finite lower bound of the support. We find the lower bound of the mean of a random variable $X$ by finding the upper mean bound of $-X$ and negating it to obtain the lower mean bound of $X$. 

First we describe the computation of Anderson's bound with $\u^{\text{And}}$ defined in Eq.~\ref{eq:And}. We compute $\beta(n)$ through Monte Carlo simulations.   $\beta(n)$ is the value such that:
\begin{align}
    \PP_{\U} (~\forall i: 1 \le i \le n, U_{(i)} \ge i/n -  \beta(n) ) = 1 - \alpha.
\end{align}
Therefore
\begin{align}
    \PP_{\U} ( \beta(n)  \ge\max_{i: 1 \le i \le n} (i/n -  U_{(i)}) ) = 1 - \alpha.
\end{align}
For each sample size $n$, we generate $L = 1,\!000,\!000$ samples $\U^j \in [0,1]^n, 1 \le j \le L$. For each sample $\U^j \in [0,1]^n$ we compute 
\begin{align*}
    \beta(n)_j = \max_{i: 1 \le i \le n} i/n - U^j_{(i)}. 
\end{align*}
Let $\beta(n)_1 \le \cdots \le \beta(n)_{L}$ be the sorted values from $L$ samples. We output $\hat{\beta}(n) = \beta(n)_{(\lceil (1-\alpha)L\rceil)}$ as an approximation of $\beta(n)$. 

For each experiment, we used $\alpha = 0.05$ unless specified otherwise. We plot the following:
    \begin{itemize}
        \item The expected value of the bounds versus the sample size. For each sample size, we draw $10,\!000$ samples of $\mathbf{x}$, compute the bound for each $\mathbf{x}$ and compute the average. 
        \item For the upper bound of the mean, we plot the $\alpha$-quantile of the bound distribution versus the sample size. For each sample size, we draw $10,\!000$ samples of $\mathbf{x}$, compute the bound for each $\mathbf{x}$ and take the $\alpha$ quantile. If the $\alpha$-quantile is below the true mean, the bound does not have guaranteed coverage.
        
        For the lower bound of the mean, we plot the $1-\alpha$-quantile of the bound distribution versus the sample size. For each sample size, we draw $10,\!000$ samples of $\mathbf{x}$, compute the bound for each $\mathbf{x}$ and take the $1-\alpha$ quantile. If the $1-\alpha$-quantile is above the true mean, the bound does not have guaranteed coverage.
        \item Coverage of the bounds. For each value of $\alpha$ from $0.02$ to $1$ with a step size of $0.02$, we draw $10,\!000$ samples of $\mathbf{x}$, compute the bound for each $\mathbf{x}$ and plot the percentage of the bounds that are greater than or equal to the true mean (denoted \emph{coverage}). If this percentage is larger than $1-\alpha$, the bound has guaranteed coverage. 
    \end{itemize}
We perform the following experiments: 
\begin{compactenum}
    \item For the case in which we know a superset $D^+$ of the distribution's support with a finite lower bound and a finite upper bound (the 2-ended support setting), we compare the following bounds: 
    \begin{itemize}
        \item Anderson's bound. 
        \item New bound with the function $T$ being Anderson's bound. 
        \item Student's $t$.
        \item Hoeffding's bound.
        \item Maurer and Pontil's bound.
    \end{itemize}
    We find an upper bound of the mean for the following distributions: 
     \begin{itemize}
        \item  $\beta(1,5)$, ${\tt uniform}(0,1)$ and $\beta(5,1)$. The known superset of the support is $[0, 1]$. The result is in Figure~\ref{fig:finite_beta1-5}.
        \item $\beta(0.5,0.5)$, $\beta(1,1)$ and $\beta(2,2)$. The known superset of the support is $[0, 1]$. The result is in Figure~\ref{fig:finite_beta1-1}. 
        \item ${\tt binomial}(10, 0.1)$, ${\tt binomial}(10, 0.5)$ and ${\tt binomial}(10, 0.9)$. The known superset of the support is the interval $[0, 10]$. The result is in Figure~\ref{fig:finite_binom}.
    \end{itemize}
    \item We also consider the case in which we want an upper bound of the mean without knowing the lower bound of the support (or to find a lower bound  without knowing an upper bound of the support). In the main paper we referred to this as the 1-ended support setting. Since Hoeffding's and Maurer and Pontil's bounds require knowing both a finite lower bound and upper bound, they are not applicable  in this setting. We compare the following bounds: 
    \begin{itemize}
        \item Anderson's bound. 
        \item New bound with $T$ being Anderson's bound. 
        \item Student's $t$
    \end{itemize}
    We address the following distributions:
     \begin{itemize}
        \item $\beta(1,5)$, ${\tt uniform}(0,1)$ and $\beta(5,1)$. The known superset of the support is $(-\infty, 1]$. We find the upper bound of the mean. The result is in Figure~\ref{fig:inf_beta1-5}. 
        \item  ${\tt binomial}(10, 0.1)$, ${\tt binomial}(10, 0.5)$ and ${\tt binomial}(10, 0.9)$. The known superset of the support is $(-\infty, 10]$. We find  the upper bound of the mean. The result is in Figure~\ref{fig:inf_binom}.
        \item ${\tt poisson}(2)$, ${\tt poisson}(10)$ and ${\tt poisson}(50)$. The known superset of the support is $[0, \infty)$. We find  the lower bound of the mean. The result is in Figure~\ref{fig:inf_poisson}.
        %   \item We know a superset of the support $[0, \infty)$. We find the lower bound of the mean of the following distributions: $lognorm(0, 0.2)$, $lognorm(0, 1)$ and $lognorm(0, 2)$.
    \end{itemize}
\end{compactenum}

\begin{figure*}[htbp]
\centering
\begin{subfigure}[b]{\textwidth}
		\centering
		\includegraphics[trim=0 0 0 0,width=0.45 \textwidth]{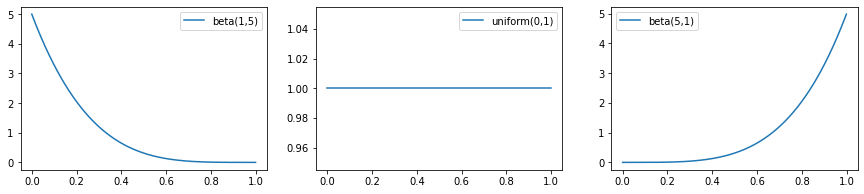}
	\caption{  The PDFs of the test distributions. }
	\end{subfigure}
\begin{subfigure}[b]{\textwidth}
		\centering
		\includegraphics[trim=0 0 0 0,width=0.8 \textwidth]{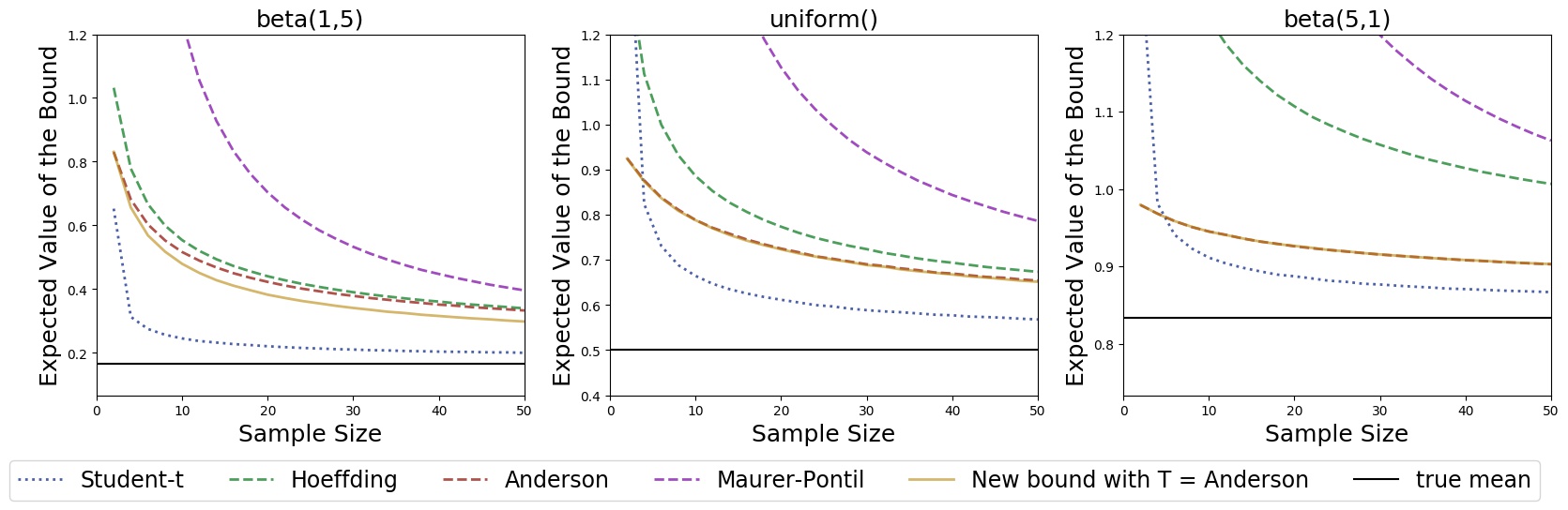}
	\caption{  Expected values of the bounds versus sample size. }
	\end{subfigure}
	\begin{subfigure}[b]{\textwidth}
		\centering
		\includegraphics[trim=0 0 0 0,width= 0.8 \textwidth]{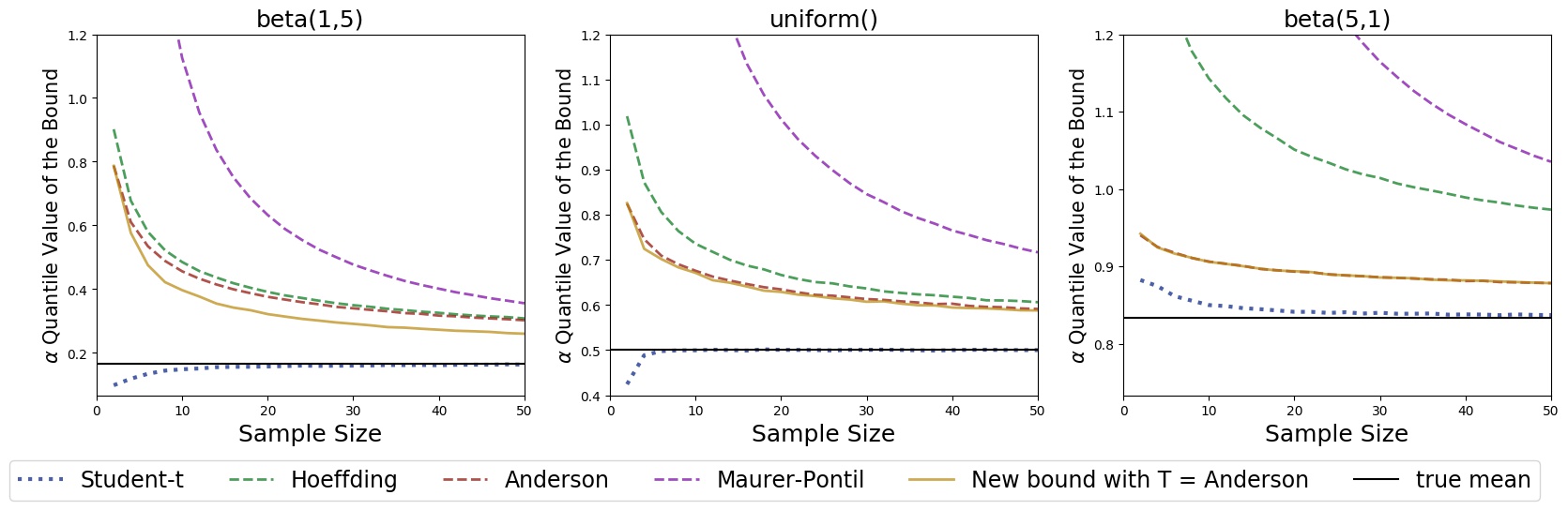}
	\caption{The $\alpha$-quantiles of bound distributions.  If the $\alpha$-quantile is below the true mean, the bound does not have guaranteed coverage. }
	\end{subfigure}
		\begin{subfigure}[b]{\textwidth}
		\centering
		\includegraphics[trim=0 0 0 0,width= 0.65 \textwidth]{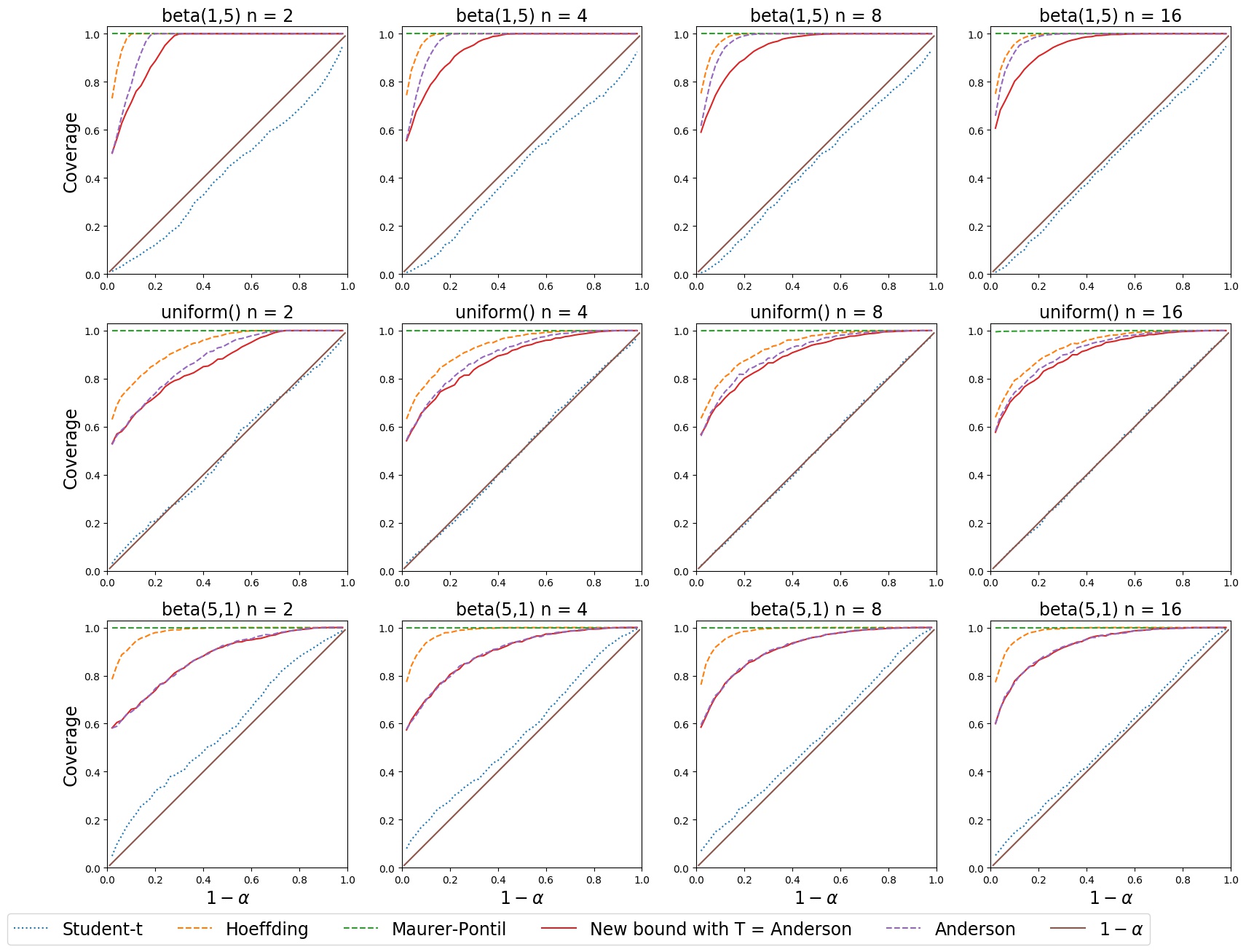}
	\caption{The coverage of the bound. If the coverage is below the line $1-\alpha$, the bound does not have guaranteed coverage.}
	\end{subfigure}
	\caption{\label{fig:finite_beta1-5} Finding the upper bound of the mean with $D^+ = [0, 1]$}
\end{figure*}

\begin{figure*}[htbp]
\centering
\begin{subfigure}[b]{\textwidth}
		\centering
		\includegraphics[trim=0 0 0 0,width=0.45 \textwidth]{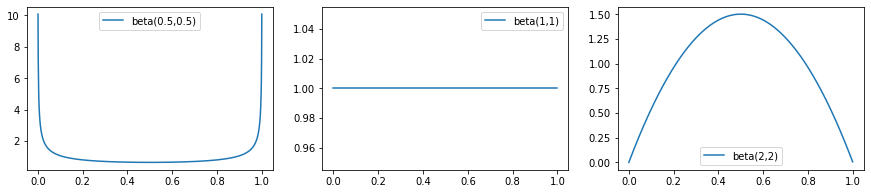}
	\caption{  The PDFs of the test distributions. }
	\end{subfigure}
\begin{subfigure}[b]{\textwidth}
		\centering
		\includegraphics[trim=0 0 0 0,width=0.8 \textwidth]{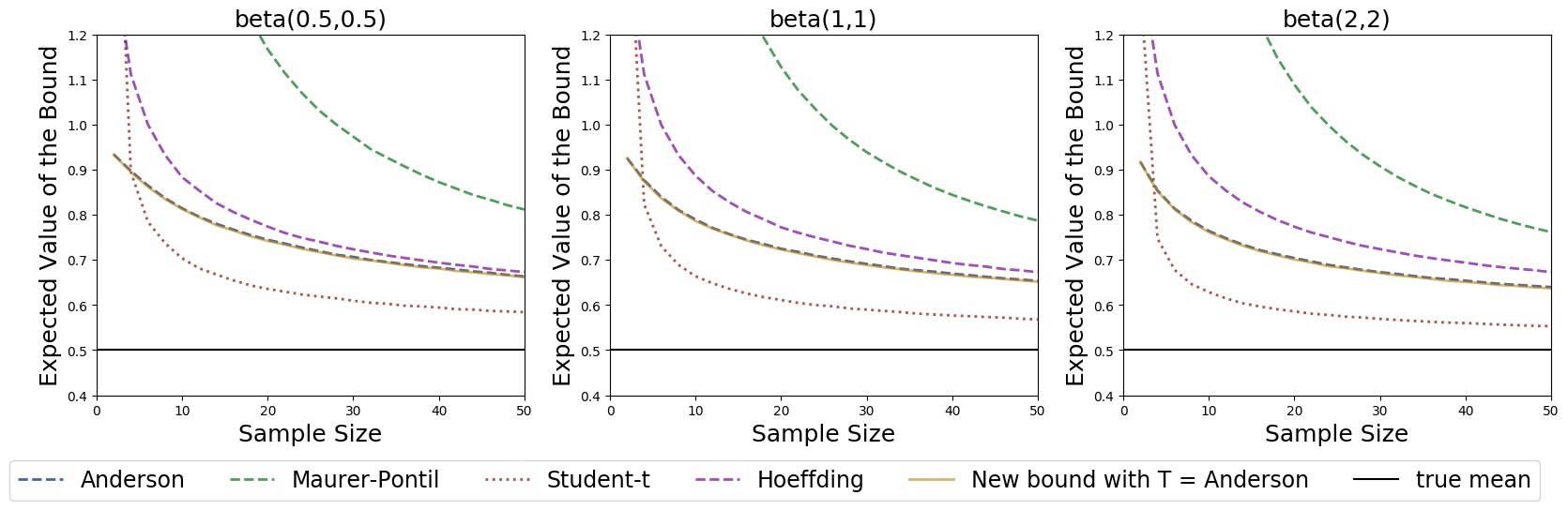}
	\caption{   Expected values of bounds versus sample size. }
	\end{subfigure}
	\begin{subfigure}[b]{\textwidth}
		\centering
		\includegraphics[trim=0 0 0 0,width= 0.8 \textwidth]{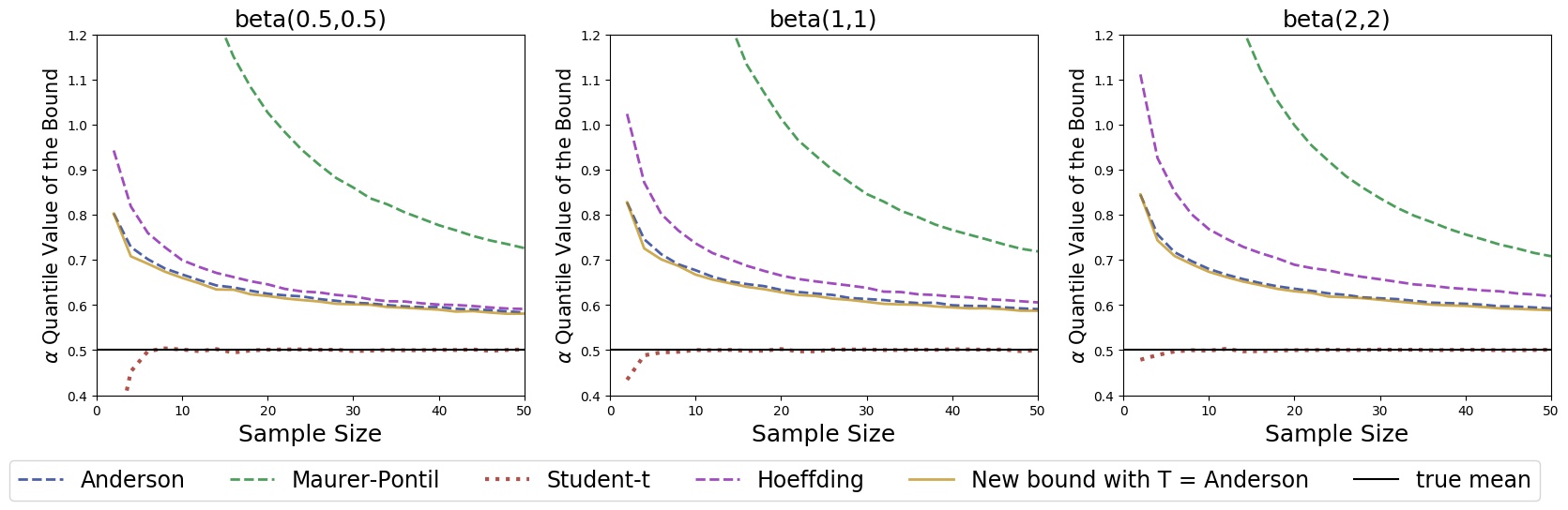}
	\caption{The $\alpha$-quantiles of bound distributions.  If the $\alpha$-quantile is below the true mean, the bound does not have guaranteed coverage. }
	\end{subfigure}
		\begin{subfigure}[b]{\textwidth}
		\centering
		\includegraphics[trim=0 0 0 0,width= 0.65 \textwidth]{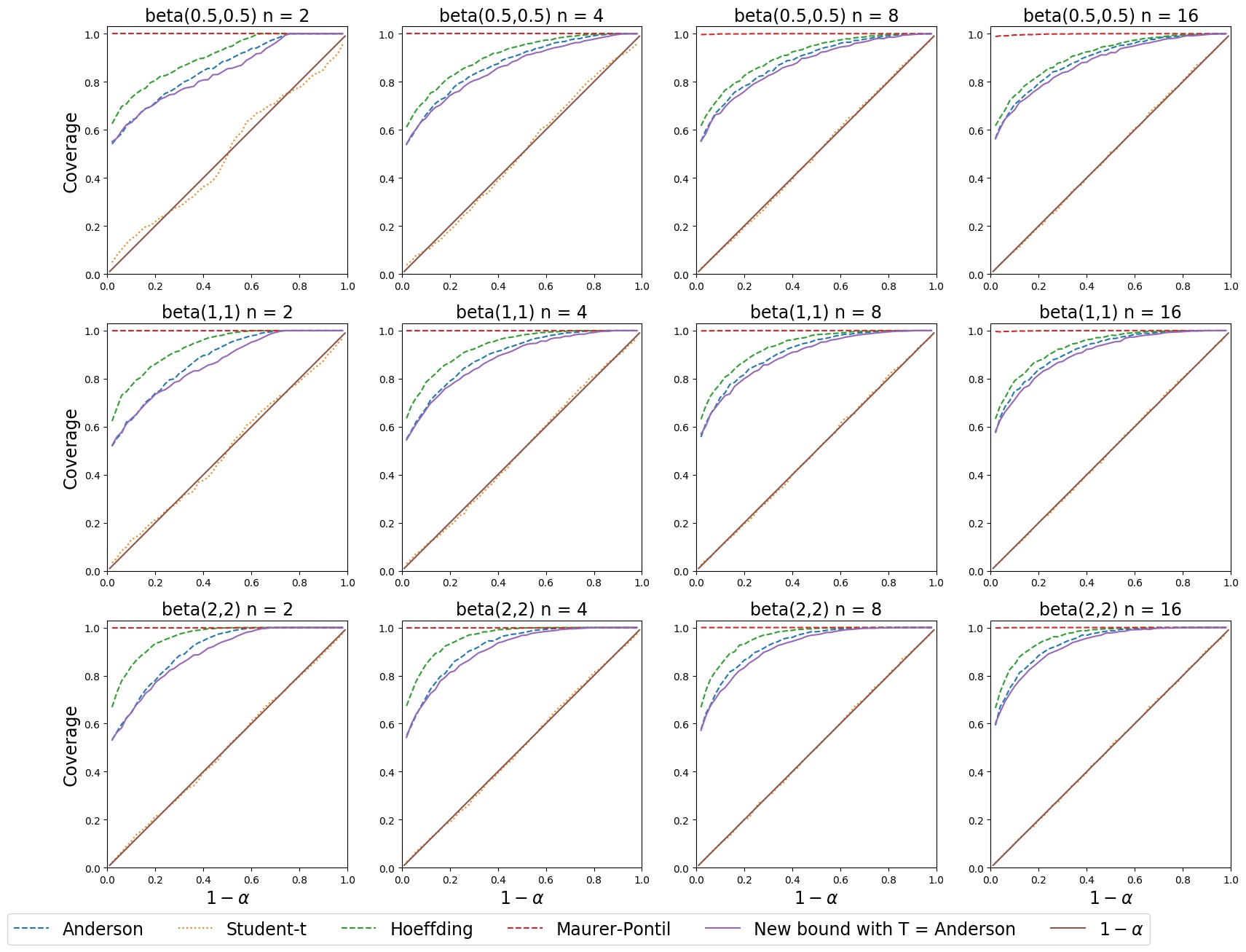}
	\caption{The coverage of the bound. If the coverage is below the line $1-\alpha$, the bound does not have guaranteed coverage.}
	\end{subfigure}
	\caption{\label{fig:finite_beta1-1} Finding the upper bound of the mean with $D^+ = [0, 1]$}
\end{figure*}

\begin{figure*}[htbp]
\centering
\begin{subfigure}[b]{\textwidth}
		\centering
		\includegraphics[trim=0 0 0 0,width=0.45 \textwidth]{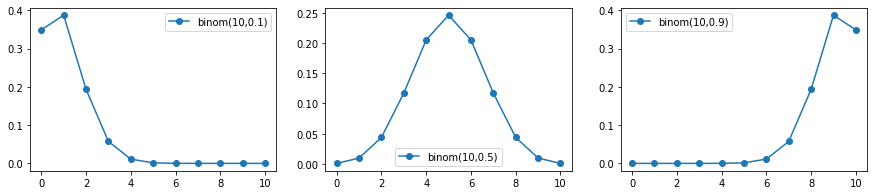}
	\caption{  The PMFs of the test distributions. }
	\end{subfigure}
\begin{subfigure}[b]{\textwidth}
		\centering
		\includegraphics[trim=0 0 0 0,width=0.8 \textwidth]{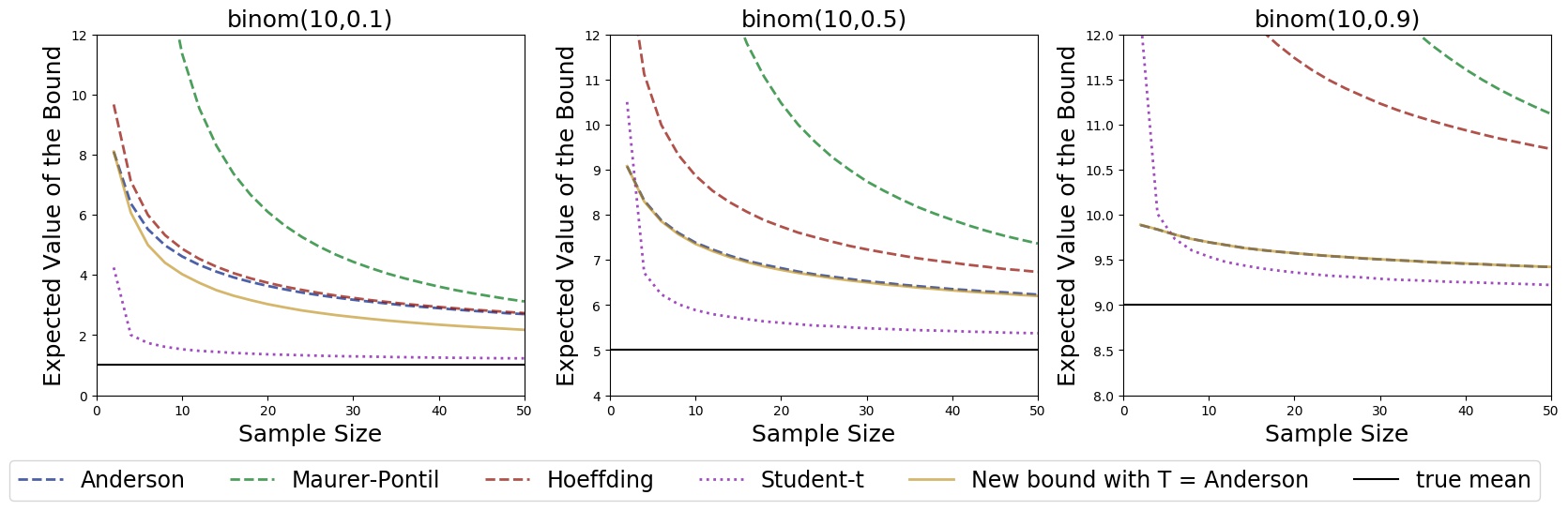}
	\caption{  Expected values of  bounds versus sample size. }
	\end{subfigure}
	\begin{subfigure}[b]{\textwidth}
		\centering
		\includegraphics[trim=0 0 0 0,width= 0.8 \textwidth]{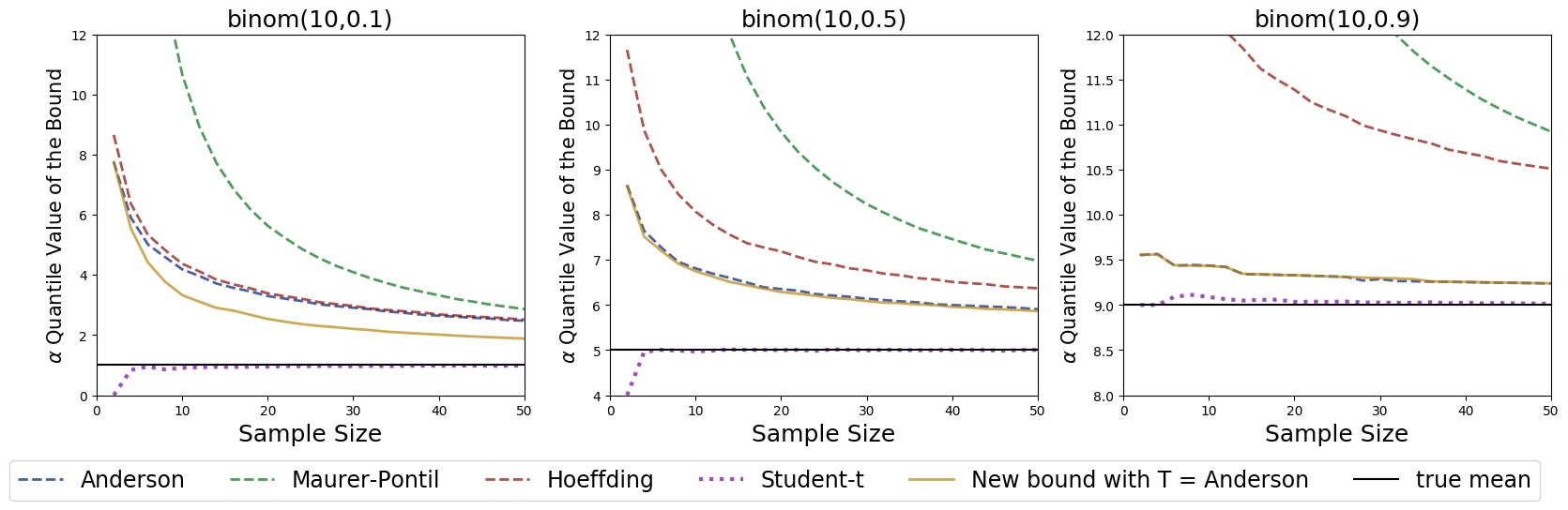}
	\caption{The $\alpha$-quantiles of  bound distributions.  If the $\alpha$-quantile is below the true mean, the bound does not have guaranteed coverage. }
	\end{subfigure}
			\begin{subfigure}[b]{\textwidth}
		\centering
		\includegraphics[trim=0 0 0 0,width= 0.65 \textwidth]{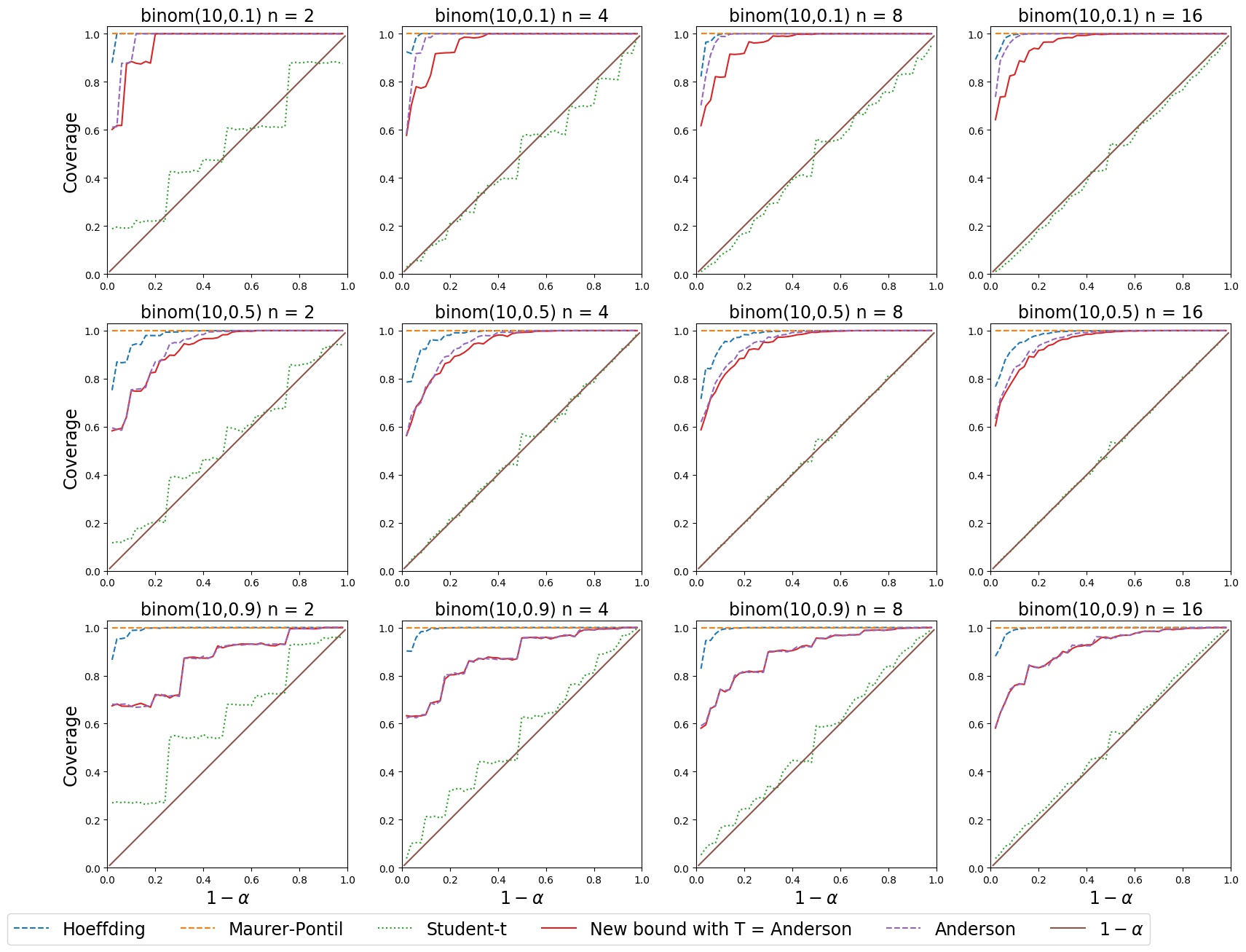}
	\caption{The coverage of the bound. If the coverage is below the line $1-\alpha$, the bound does not have guaranteed coverage.}
	\end{subfigure}
	\caption{\label{fig:finite_binom} Finding the upper bound of the mean with $D^+ = [0,10]$}
\end{figure*}

\begin{figure*}[htbp]
\centering
\begin{subfigure}[b]{\textwidth}
		\centering
		\includegraphics[trim=0 0 0 0,width=0.45 \textwidth]{figs/finite_beta1-5_pdf.png}
	\caption{  The PDFs of the test distributions. }
	\end{subfigure}
\begin{subfigure}[b]{\textwidth}
		\centering
		\includegraphics[trim=0 0 0 0,width=0.75 \textwidth]{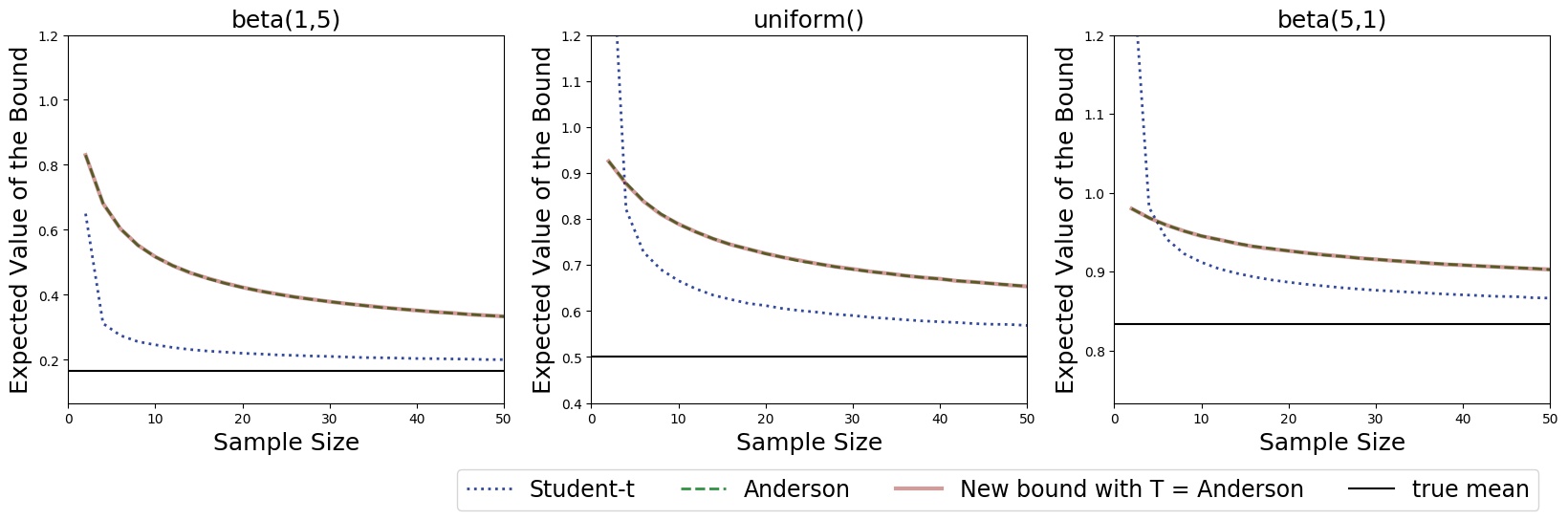}
	\caption{  The expected value of the bounds. }
	\end{subfigure}
	\begin{subfigure}[b]{\textwidth}
		\centering
		\includegraphics[trim=0 0 0 0,width= 0.75 \textwidth]{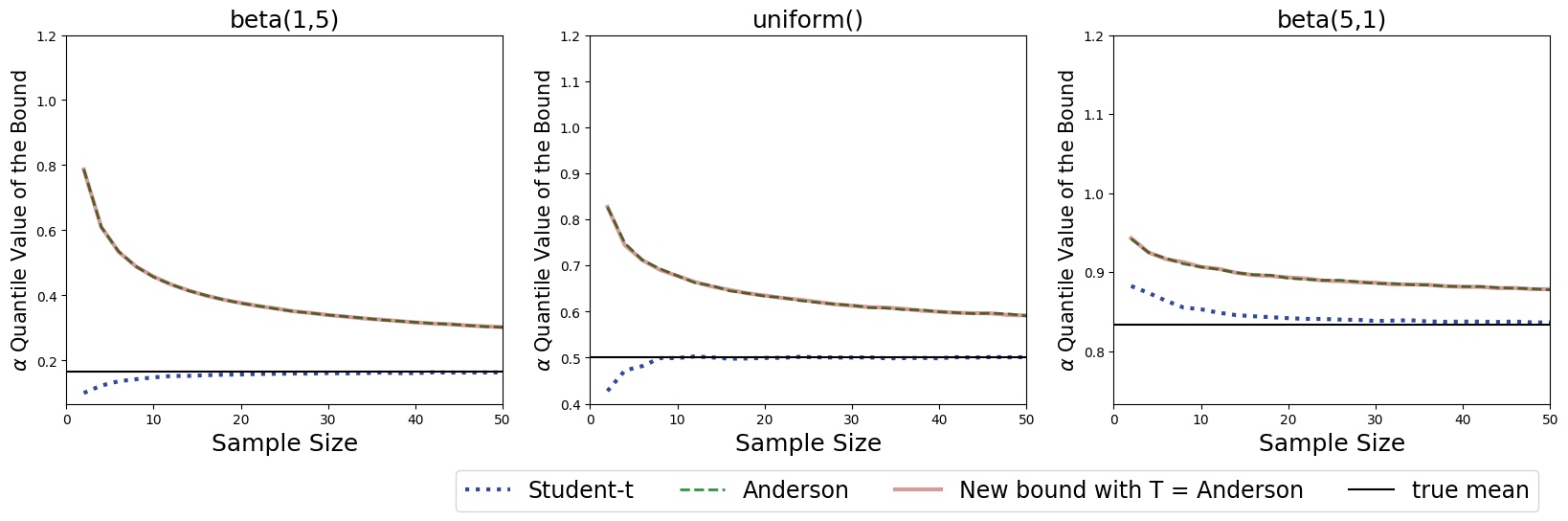}
	\caption{The $\alpha$-quantile of the bound distribution.  If the $\alpha$-quantile is below the true mean, the bound does not have guaranteed coverage. }
	\end{subfigure}
				\begin{subfigure}[b]{\textwidth}
		\centering
		\includegraphics[trim=0 0 0 0,width= 0.6 \textwidth]{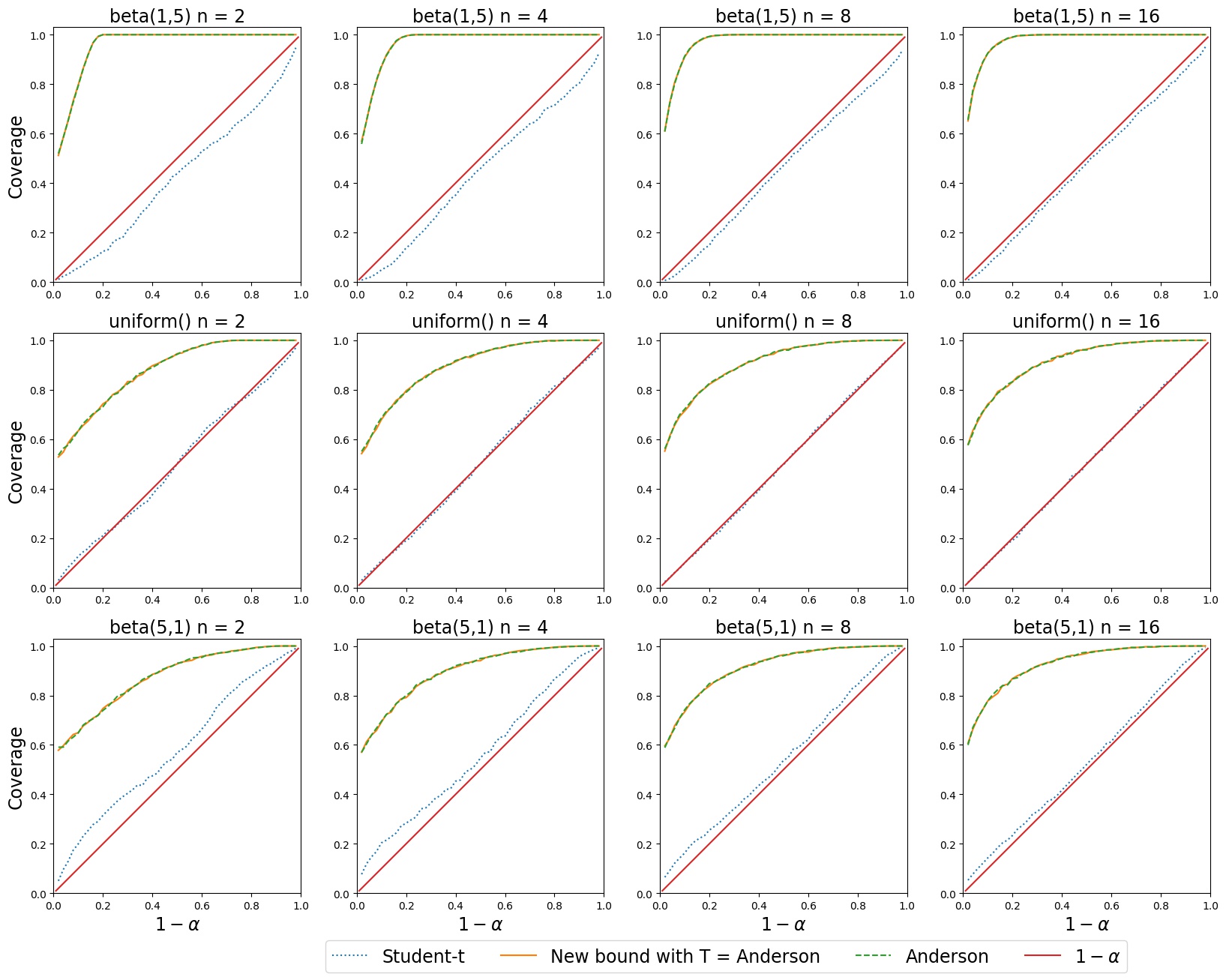}
			\caption{The coverage of the bound. If the coverage is below the line $1-\alpha$, the bound does not have guaranteed coverage.}
	\end{subfigure}
	\caption{\label{fig:inf_beta1-5} Finding the upper bound of the mean with $D^+ = (-\infty, 1]$}
\end{figure*}

\begin{figure*}[htbp]
\centering
\begin{subfigure}[b]{\textwidth}
		\centering
		\includegraphics[trim=0 0 0 0,width=0.45 \textwidth]{figs/binom_pmf.png}
	\caption{  The PMF of the test distributions. }
	\end{subfigure}
\begin{subfigure}[b]{\textwidth}
		\centering
		\includegraphics[trim=0 0 0 0,width= 0.75\textwidth]{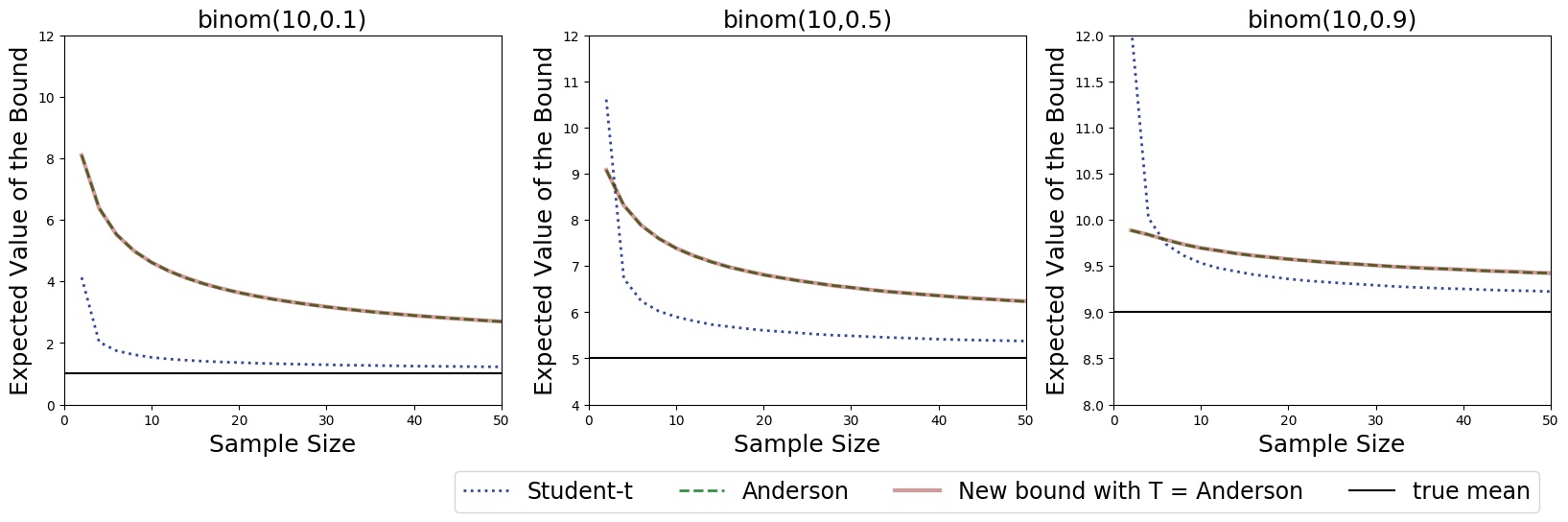}
	\caption{ The expected value of the bounds. }
	\end{subfigure}
	\begin{subfigure}[b]{\textwidth}
		\centering
		\includegraphics[trim=0 0 0 0,width=0.75 \textwidth]{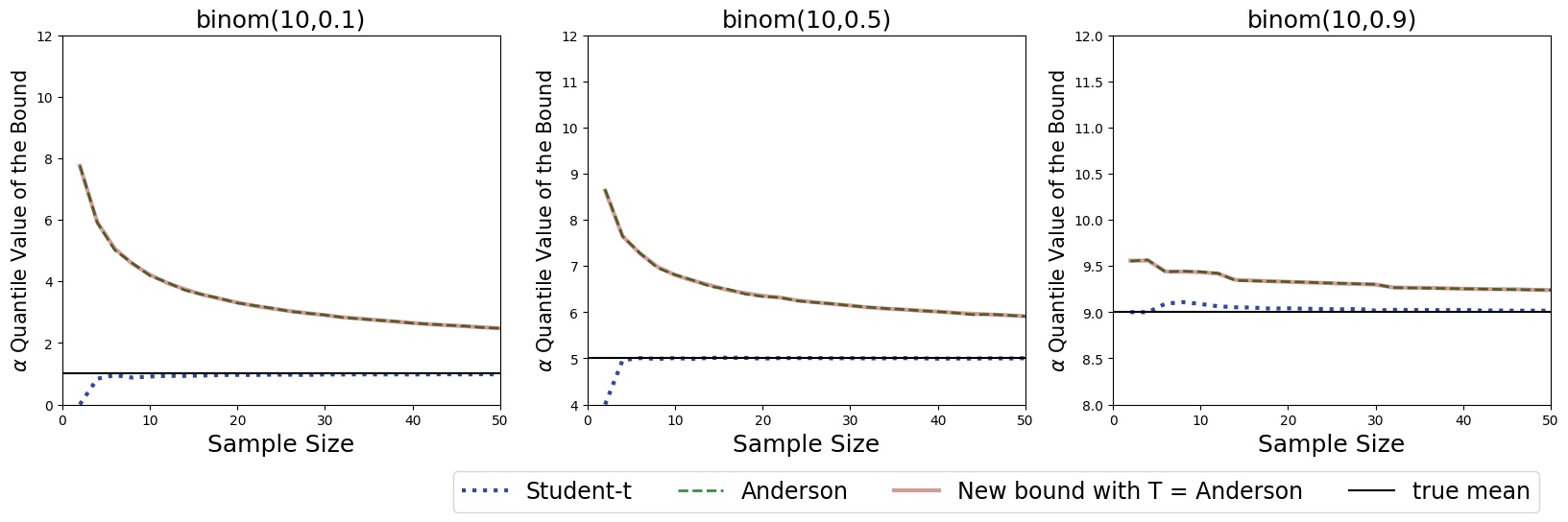}
	\caption{ The $\alpha$-quantile of the bound distribution.  If the $\alpha$-quantile is below the true mean, the bound does not have guaranteed coverage. }
	\end{subfigure}
				\begin{subfigure}[b]{\textwidth}
		\centering
		\includegraphics[trim=0 0 0 0,width= 0.6 \textwidth]{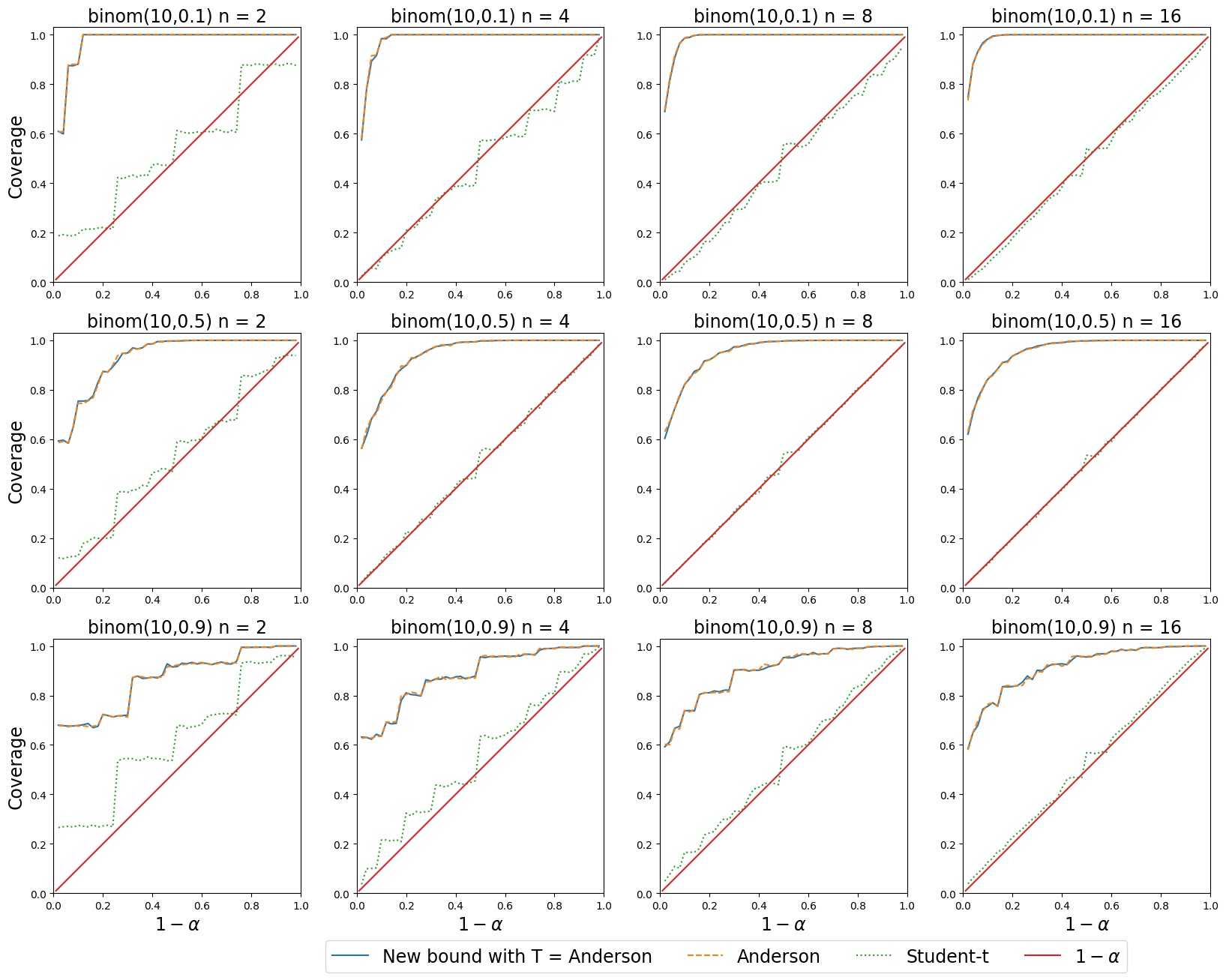}
		\caption{The coverage of the bound. If the coverage is below the line $1-\alpha$, the bound does not have guaranteed coverage.}
	\end{subfigure}
	\caption{\label{fig:inf_binom} Finding the upper bound of the mean with $D^+ = (-\infty, 10]$}
\end{figure*}
\begin{figure*}[htbp]
\centering
\begin{subfigure}[b]{\textwidth}
		\centering
		\includegraphics[trim=0 0 0 0,width=0.45 \textwidth]{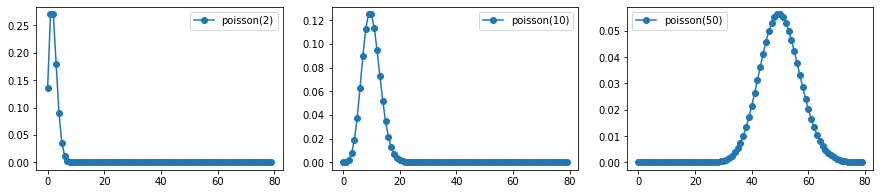}
	\caption{  The PMF of the test distributions. }
	\end{subfigure}
\begin{subfigure}[b]{\textwidth}
		\centering
		\includegraphics[trim=0 0 0 0,width=0.72 \textwidth]{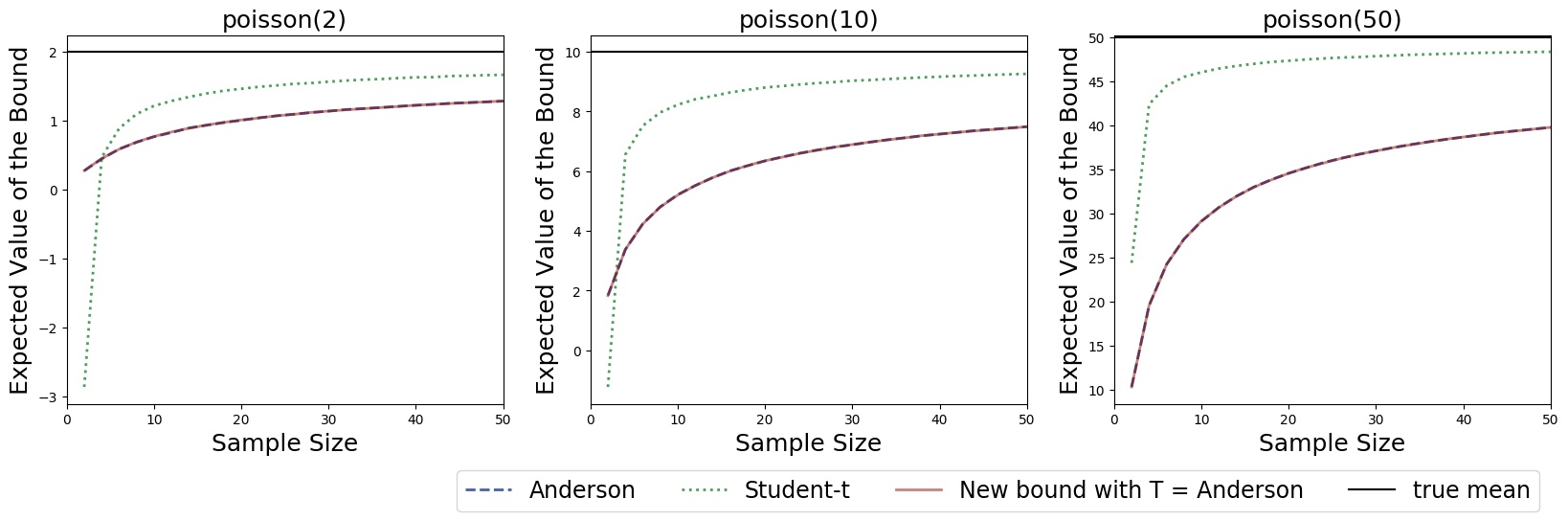}
	\caption{  The expected value of the bounds. }
	\end{subfigure}
	\begin{subfigure}[b]{\textwidth}
		\centering
		\includegraphics[trim=0 0 0 0,width=0.72 \textwidth]{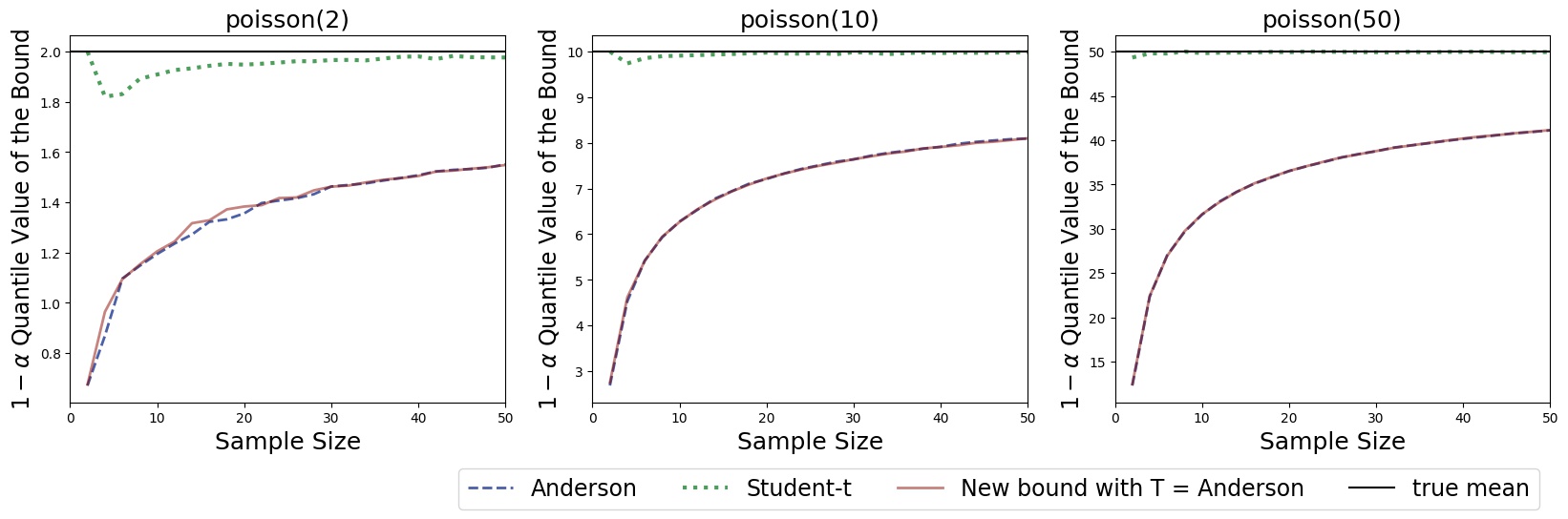}
	\caption{ The $1- \alpha$-quantile of the bound distribution.  If the $1 - \alpha$-quantile is above the true mean, the  bound does not have guaranteed coverage. }
	\end{subfigure}
				\begin{subfigure}[b]{\textwidth}
		\centering
		\includegraphics[trim=0 0 0 0,width= 0.6 \textwidth]{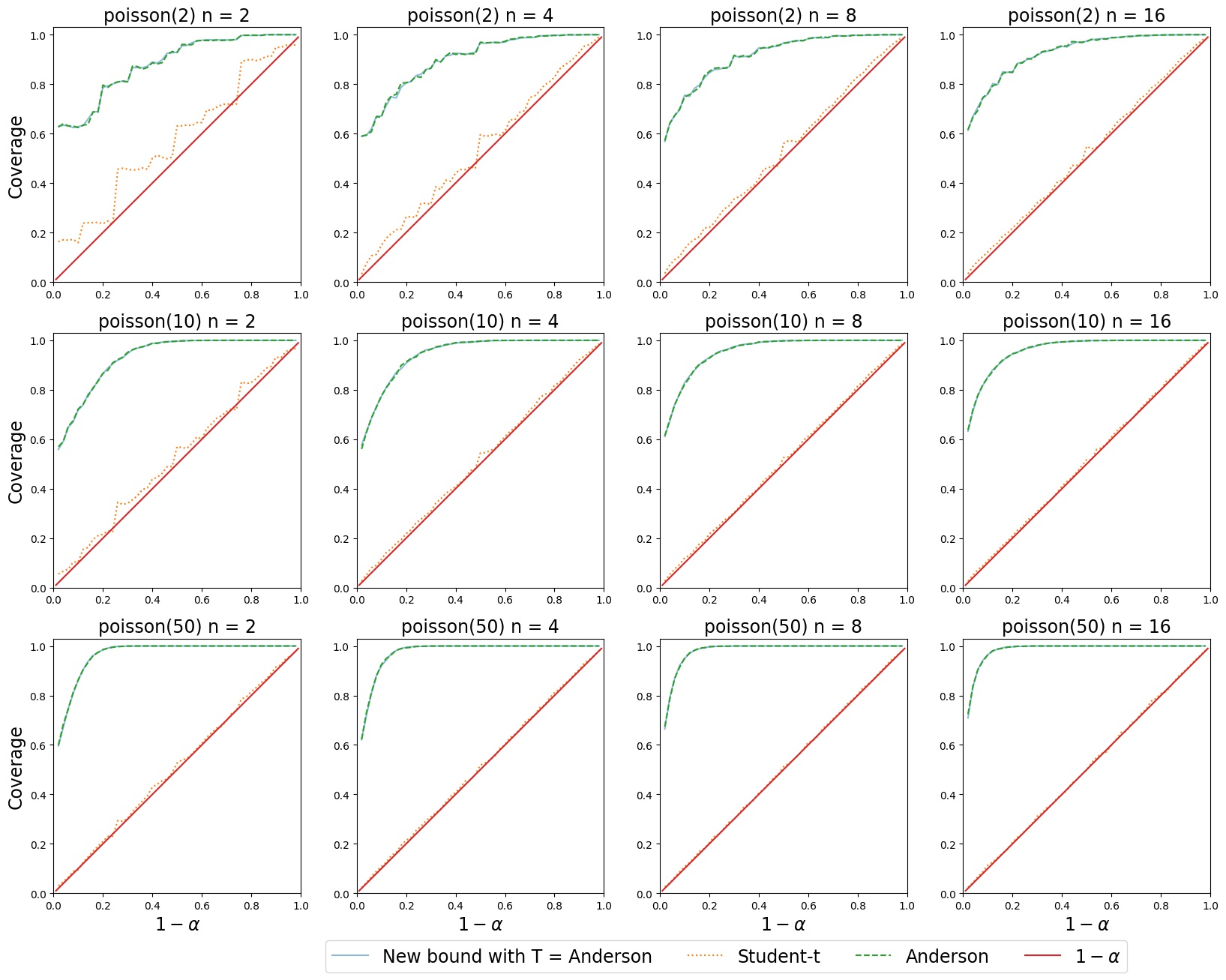}
	\caption{The coverage of the bound. If the coverage is below the line $1-\alpha$, the bound does not have guaranteed coverage.}
	\end{subfigure}
	\caption{\label{fig:inf_poisson} Finding the lower bound of the mean with $D^+ = [0, \infty)$ }
\end{figure*}

All the experiments confirm that our bound has guaranteed coverage and is equal to or tighter than Anderson's and Hoeffding's. 

From the experiments, our upper bound performs the best in distributions that are skewed right (respectively, our lower bound will perform the best in distributions that are skewed left), when we know a tight lower bound and upper bound of the support. 

\section{Discussion on Section~\ref{sec:intro}: Skewed Sample Mean Distribution with $n = 80$}
\label{apx:intro}
In this section, as noted in Section~\ref{sec:intro}, we show a log-normal distribution where the sample mean distribution is visibly skewed when $n = 80$ (Figure~\ref{fig:motivating_example}). Student's $t$ is not a good candidate in this case because the sample mean distribution is not approximately normal. This example is a variation on the one provided by~\citet{Frost2021}. 

While the log-normal distribution is an extreme example of skew, this example illustrates the danger of assuming the validity of arbitrary thresholds on the sample size, such as the traditional threshold of $n=30$, for using the Student's $t$ method. Clearly there are cases where such a threshold, and even much larger thresholds, are not adequate.
\begin{figure}[htbp]
		\centering
		\includegraphics[width = 0.45\textwidth]{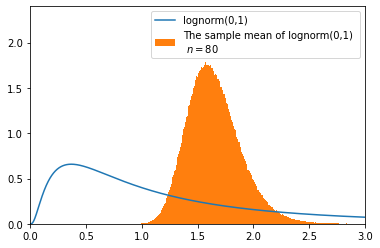}
	\caption{\label{fig:motivating_example} The PDFs of $lognorm(0,1)$ and the sample mean distribution of $lognorm(0,1)$. The sample mean distribution of $lognorm(0,1)$ is visibly skewed when the sample size $n = 80$. }
\end{figure}
\section{Proof of Section~\ref{sec:main_result}} 
\label{apx:proof_guarantee}
We restate the lemma and theorem statements for convenience. 
\begin{lemma}[Lemma~\ref{lem:CDF_vs_U}]
Let $X$ be a random variable with CDF $F$ and $Y \definedas F(X)$, known as the probability integral transform of $X$. Let $U$ be a uniform random variable on $[0,1]$. Then for any $0 \le y \le 1$, 
\begin{align}
    \PP(Y \le y) \le \PP(U \le y).
\end{align}
If $F$ is continuous, then $Y$ is uniformly distributed on $(0,1)$. 
\end{lemma}
\begin{proof}
Let $F^{-1}(y) = \inf\{x : F(x) \ge y\} $ for $0 < y < 1$ and $U$ be an uniform random variable on $(0,1)$. Since $F$ is non-decreasing and right-continuous, $F(F^{-1}(y)) \ge y$ . By \citet{10.1137/1036146}, $F^{-1}(U)$ has CDF $F$. For $0 < y < 1$, then: 
\begin{align}
    \PP( Y \le y ) &= \PP( F(X) \le y) \\
    &= \PP( F(F^{-1}(U)) \le y  )  \\
    &\le \PP(U \le y)  \\
    &= y. 
\end{align}

If $F$ is continuous, \citet{10.1137/1036146} shows that $Y$ is uniformly distributed on $(0,1)$. 
\end{proof}
% \begin{proof}
% Since $\PP(U \le y) = y$, we will show that $\PP(Y \le y) \le y$. 

% We will first show that if $F(x) \le y$, then  $x \le \sup\{z: F(z) \le y\}$. Suppose that $x > \sup\{z: F(z) \le y\}$. Then, $F(x) > y$. Therefore, 
% \begin{align}
% F(x) \le y \text{ implies } x \le \sup\{z: F(z) \le y\}. \label{eq:helper5}
% \end{align}
% Now we have
% \begin{align}
%     \PP( Y \le y ) &= \PP( F(X) \le y) \\
%     &\le \PP( X \le \sup\{z: F(z) \le y\} )  \\
%     &= F(z^*) \text{ where } z^* = \sup\{z: F(z) \le y\} \\
%     &\le y. 
% \end{align}
% If $F$ is continuous, \citet{10.1137/1036146} shows that $Y$ is uniformly distributed on $(0,1)$. 
% \end{proof}
%%%%%%%%% ICML version
% \my{Begin of ICML version}
\begin{lemma}[Lemma~\ref{lem:helper2}]
Let $F$ and $G_{\x,\l}$ be two CDF functions such that $\forall x \in \mathcal{R}$, $F(x) \ge G_{\x,\l}(x)$. Let $\mu_F$ and $\mu_G$ denote the means of $F$ and $G_{\x,\l}$. Then:
\begin{align}
  \mu_F \le \mu_G. 
\end{align}
\end{lemma}
\begin{proof}

Let $x_{(0)} \definedas - \infty$ and $x_{(n+1)} \definedas s_D$.
Then: 
\begin{align}
    \mu_F &=  \int x \;dF(x) \\
    &= \sum_{i=1}^{n+1} \int_{x_{(i-1)}}^{x_{(i)}} x  \;dF(x) \\
    &\le \sum_{i=1}^{n+1} \int_{x_{(i-1)}}^{x_{(i)}} x_{(i)}  \;dF(x) \\
    &=  \sum_{i=1}^{n+1}  x_{(i)} (F(x_{(i)})- F(x_{(i-1)})) \\
    &= s_D - \sum_{i=1}^n F(x_{(i)})(x_{(i+1)}- x_{(i)}) \\
    &\le s_D - \sum_{i=1}^n G(x_{(i)})(x_{(i+1)}- x_{(i)}) \\
    &= \mu_G 
\end{align}

\end{proof}
% \my{End of ICML version}

% \my{Begin of Arxiv version}
% \begin{lemma}[Lemma~\ref{lem:helper2}]
% Let $F$ and $G$ be 2 CDF functions such that $\forall x \in \mathcal{R}$, $F(x) \ge G(x)$. Let $\mu_F$ and $\mu_G$ denote the mean of $F$ and $G$. Then
% \begin{align}
%   \mu_F \le \mu_G. 
% \end{align}
% \end{lemma}
% \begin{proof}
% We present the proof with the notation for when $F$ and $G$ are continuous. The proof when $F$ and $G$ have discontinuity is a trivial extension. From \citep{Mean_Wikipedia}:
% \begin{align}
%     \mu_F &= \int_{0}^{\infty} (1 - F(x)) dx - \int_{-\infty}^0 F(x) dx \\ &\le \int_{0}^{\infty} (1 - G(x)) dx - \int_{-\infty}^0 G(x) dx \\
%     &= \mu_G.
% \end{align}
% \end{proof}
% \my{end of arxiv version}
\begin{lemma}[Lemma~\ref{lem:helper}]
Let $\Z$ be a random sample of size $n$ from $F$. Let $\U=U_1,...,U_n$ be a sample of size $n$ from the continuous uniform distribution on $[0,1]$. For any function $T: D^n \rightarrow R$ and any $\x \in D^n$,
\begin{align}
    \PP_{\Z} (T(\Z) \le T(\x)) \le \PP_{\U} (b(\x, \mathbf{U}) \ge \mu  ).
\end{align}
\end{lemma}
\begin{proof}

Let $\cup$ denote the union of events and $\{\}$ denote an event. Let $\Z$ be a sample from $F$. Then for any sample $\x$:
% \textcolor{red}{PST: See comment in overleaf.}
\begin{align}
    & \PP_{\Z} (T(\Z) \le T(\x)) \\% \label{eq:z_space}\\
    & = \PP_{\Z}(\Z \in \mathbb{S}(\x)) 
    \\ 
    & = \PP_{\Z} (\cup_{\y \in \mathbb{S}(\x)} \{\Z = \y \})
    \\ 
    & \le \PP_{\Z} (\cup_{\y \in \mathbb{S}(\x)} \Z \preceq \y) \text{ because $\Z = \y$ implies $\Z \preceq \y$ } \\%\label{eq:z_space_union}
    &\le \PP_{\Z} (\cup_{\y \in \mathbb{S}(\x)} \{F(\Z) \preceq F(\y)\})  \label{eq:apx_u_space}
    \end{align}
    because $F$ is non-decreasing, so $Z_{(i)} \le y_{(i)}$ implies $F(Z_{(i)}) \le F(y_{(i)})$. Let $U_1, \cdots, U_n$ be $n$ samples from the uniform distribution on $(0,1)$.   From Lemma~\ref{lem:CDF_vs_U}, for any $u\in (0,1)$, $\PP(F(Z_i) \le u) \le \PP(U_i \le u)$. 
Therefore
    \begin{align}
    &\PP_{\Z} (\cup_{\y \in \mathbb{S}(\x)} \{F(\Z) \preceq F(\y)\})  \\
    &\le \PP_{\U} (\cup_{\y \in \mathbb{S}(\x)} \{\U \preceq F(\y) \} ). 
 \label{eq:apx_u_space_bound}
    \end{align}
     Recall that $m_D(\y, \U )=  s_D - \sum_{i = 1}^{n} U_{(i)}(y_{(i+1)} - y_{(i)})$ where $\forall i, y_{(i+1)} - y_{(i)} \ge 0$. Therefore if $\forall i, U_{(i)} \le F(y_{(i)})$ then $m_D(\y, \U ) \ge m_D(\y, F(\y))$: 
    \begin{align}
    &\PP_{\U} (\cup_{\y \in \mathbb{S}(\x)} \{\U \preceq F(\y) \} ) \\
     &\le \PP_{\U} (\cup_{\y \in \mathbb{S}(\x)}  \{ m_D(\y, \mathbf{U}) \ge m_D(\y, F(\y) ) \}), \text{by Lemma~\ref{lem:helper2}} \\%\label{eq:u_space_bound1} \\
    &\le \PP_{\U} (\cup_{\y \in \mathbb{S}(\x)} \{  m_D(\y, \mathbf{U}) \ge \mu \} ),   \text{by Lemma~\ref{lem:helper2}} \\%\label{eq:u_space_bound2} \\
    &\le \PP_{\U} (\sup_{\y \in \mathbb{S}(\x)}   m_D(\y, \mathbf{U}) \ge \mu  )  \label{eq:sup} \\
    &= \PP_{\U} (b(\x, \mathbf{U}) \ge \mu  ) \label{eq:apx_final}
\end{align}
The inequality in Eq.~\ref{eq:sup} is because if there exists $\y \in \mathbb{S}(\x)$ such that $m_D(\y, \mathbf{U}) \ge \mu$, then $\sup_{\y \in \mathbb{S}(\x)}   m_D(\y, \mathbf{U}) \ge \mu$. Therefore the event $\cup_{\y \in \mathbb{S}(\x)} \{  m_D(\y, \mathbf{U}) \ge \mu \}$ is a subset of the event $\sup_{\y \in \mathbb{S}(\x)}   m_D(\y, \mathbf{U}) \ge \mu$, and Eq.~\ref{eq:sup} follows. 

From Eqs.~\ref{eq:apx_u_space},~\ref{eq:apx_u_space_bound} and Eq.~\ref{eq:apx_final}: 
\begin{align}
    \PP_{\Z} (T(\Z) \le T(\x)) \quad \le \quad\PP_{\U} (b(\x, \mathbf{U}) \ge \mu  ).
\end{align}
\end{proof}

\section{Discussion on Section~\ref{sec:computation}: Monte Carlo Convergence} 
\label{apx:mc}
In Section~\ref{sec:computation}, we discussed the use of Monte Carlo sampling of the induced mean function $b(\x,\U)$ via sampling of the uniform random variable $U$, to approximate the $1-\alpha$ quantile of $b(\x,\U)$. Let $\hat{q}_{\ell}$ denote the output of the Monte Carlo algorithm (Algorithm~\ref{alg:monte_carlo}) using $\ell$ Monte Carlo samples. In this section we show that our estimator converges to the true quantile as the number of Monte Carlo samples grows, and, given a desired threshold $\epsilon$, we can compute an upper bound at most $Q(1-\alpha, b(\x,\U)) + \epsilon$ with guaranteed coverage. 
\begin{theorem}
\label{thm:MC_final}
Let $\epsilon > 0$. Let $\gamma = \min \left ( \alpha, \left( \frac{\epsilon}{3 (s_D - \sup_{\z \in \mathbb{S}(\x))}z_{(1)})} \right)^n \right)$. Use  $\ell = \left \lceil \frac{-\ln(\gamma/2)}{2} \left(\frac{3(s_D - \sup_{\z \in \mathbb{S}(\x)} z_{(1)})}{\epsilon}\right)^n \right \rceil$ Monte Carlo samples to compute $Q(1 - \alpha + \gamma, b(\x,\U))$ using Algorithm~\ref{alg:monte_carlo}. Let $\hat{q}_{\ell}$ be the output of the algorithm. We output $\hat{q}_{\ell} + \epsilon/3$ as the final estimator. 

Then with probability at least $1-\alpha$: 
\begin{align}
\mu \le \hat{q}_{\ell} + \epsilon/3 \le Q(1-\alpha, b(\x,\U)) + \epsilon.  \end{align}
\end{theorem}
To prove Theorem~\ref{thm:MC_final}, we first show some lemmas. 

The Monte Carlo approximation error is quantified in the following lemma due to \citet{GVK024353353}. Let $F(m-) \definedas \lim_{x \rightarrow m^-} F(x)$. 
\begin{lemma}[Theorem 2.3.2 in \citet{GVK024353353}]
\label{lem:MC}
Let $0 < p < 1$. If $Q(p, M)$ is the unique solution $m$ of $F(m-) \le p \le F(m)$, then for every $\epsilon > 0$, 
\begin{align}
\PP(   | \mathbf{M}_{\lceil pl\rceil} -   Q(p, M)|  > \epsilon ) \le 2e^{-2l\delta} , 
\end{align}
where $\mathbf{M}_k$ denotes the $k$-th order statistic of the sample $\M$ and 
$$\delta = \min\left(p - F(Q(p, M) - \epsilon), F(Q(p, M) + \epsilon) - p\right).$$ 
\end{lemma}
Note that when the condition that $Q(p, M)$ is the unique solution $m$ of $F(m-) \le p \le F(m)$ is satisfied, $\delta > 0$. Let $M \definedas b_{D, T}(\x, \U) \in [0,1]$. 
In Lemma~\ref{lem:MCMC} we will show that the  CDF of $M$ satisfies the condition in Lemma~\ref{lem:MC}. Therefore the error incurred by computing the bound via Monte Carlo sampling can be decreased to an arbitrarily small value by choosing a large enough number of Monte Carlo samples $l$. The Monte Carlo estimation of $b^\alpha_{D^+,T}(\x)$ where $D^+ = [0,1]$ is presented in Algorithm~\ref{alg:monte_carlo}. 

We will show that for any $\x$, for any $T$, for any $p \in (0,1)$, $F_M(m-) \le p \le F_M(m)$ has a unique solution by showing that for any $\x$ and $T$, $F_M$ is strictly increasing on its support. To do so, for any $c_1, c_2$ in the support such that $c_1 < c_2 $ we will show that
\begin{align}
  F_M(c_2) - F_M(c_1) > 0 .
\end{align}

\begin{lemma}
\label{lem:MCMC}
Let $M \definedas b(\x, \U)$. Let $F_M$ be the CDF of $M$. 

For any $\x$, for any scalar function $T$, either:
\begin{compactenum}
\item $M$ is a constant, or
\item For any $c_1, c_2$ such that $0 \le c_1 < c_2 \le 1$, 
\begin{align}
    F_M(c_2) - F_M(c_1) \ge \left( \frac{c_2 - c_1}{s_D - \sup_{\z \in \mathbb{S}(\x)} z_{(1)}}\right)^n > 0.
\end{align}
\end{compactenum}
\end{lemma}
\begin{proof}
%  We will show that for any $\x$, for any $T$, for any $p \in (0,1)$, $F_M(m-) \le p \le F_M(m)$ has a unique solution by showing that for any $\x$ and $T$, $F_M$ is strictly increasing on its support. To do so, for any $c_1, c_2$ in the support such that $c_1 < c_2 $ we will show that
% \begin{align}
%   F_M(c_2) - F_M(c_1) > 0 .
% \end{align}

Recall the definition of the induced mean as 
\begin{align}
  b(\x, \l) &= \sup_{\z \in \mathbb{S}(\x)} \sum_{i=1}^{n+1} z_{(i)}(\ell_{(i)} - \ell_{(i-1)}), \\
  &=  \sup_{\z \in \mathbb{S}(\x)} s_D - \sum_{i=1}^n \ell_{(i)}(z_{(i+1)}- z_{(i)}),
\end{align}
where $\ell_{(0)} \definedas 0$, $\ell_{(n+1)} \definedas 1$ and $z_{(n+1)} \definedas s_D$.

We now find the support of $M$. Let $\phi \definedas \sup_{\z \in \mathbb{S}(\x)} z_{(1)} $. We will show that for any $\u$ where $0 \le u_i \le 1$, we have $ \phi \le b(\x, \u) \le s_D$, and therefore the support of $M$ is a subset of $[\phi, s_D]$. We have
\begin{align}
  b(\x, \u) &= \sup_{\z \in \mathbb{S}(\x)} s_D - \sum_{i=1}^n u_{(i)}(z_{(i+1)}- z_{(i)}) \\
  &\le \sup_{\z \in \mathbb{S}(\x)} s_D - \sum_{i=1}^n 0 (z_{(i+1)}- z_{(i)}) \\
 & = s_D.
\end{align}
Similarly we have 
\begin{align}
  b(\x, \u) &= \sup_{\z \in \mathbb{S}(\x)} s_D - \sum_{i=1}^n u_{(i)}(z_{(i+1)}- z_{(i)}) \\
  &\ge \sup_{\z \in \mathbb{S}(\x)} s_D - \sum_{i=1}^n 1 (z_{(i+1)}- z_{(i)}) \\
 & = \sup_{\z \in \mathbb{S}(\x)} s_D - (z_{(n+1)} - z_{(1)}) \\
 & = \sup_{\z \in \mathbb{S}(\x)}z_{(1)} \\
 &= \phi.
\end{align}
Therefore $M = b(\x, \U) \in [\phi, s_D]$. 
We consider two cases: where $\phi = s_D$ and where $\phi < s_D$. 

{\bf Case $1$: $\phi = s_D$.}  

Then for all $i$, $\sup_{\z \in \mathbb{S}(\x)} z_{(i)} \ge \phi = s_D$. Since $z_{(i)} \le s_D$, we have for all $i$, $\sup_{\z \in \mathbb{S}(\x)} z_{(i)} = s_D$. Therefore
\begin{align}
  b(\x, \l) &= \sup_{\z \in \mathbb{S}(\x)} \sum_{i=1}^{n+1} z_{(i)}(\ell_{(i)} - \ell_{(i-1)}) \\
  &= \sum_{i=1}^{n+1} s_D (\ell_{(i)} - \ell_{(i-1)}) \\
  &= s_D.
\end{align}
Therefore $M = b(\x,\U)$ is a constant $s_D$, and the $1-\alpha$ quantile of $M$ is $s_D$. 

{\bf Case $2$: $\phi < s_D$.} 

Let $c_1, c_2 \in \mathcal{R}$ be such that $\phi \le c_1 < c_2 \le s_D$. We will now show that 
\begin{align}
  F_M(c_2) - F_M(c_1) > 0. 
\end{align}

Let $v \definedas \frac{s_D - c_2}{s_D - \phi}$ and $w \definedas \frac{s_D - c_1}{s_D - \phi}$. If $\phi \le c_1 < c_2 \le s_D$ then $v < w$ and $v, w \in [0,1]$. 

Let $\v \definedas (v_1, \cdots, v_n)$ and $\w \definedas (w_1, \cdots, w_n)$ where $\forall i, v_i = v$ and $w_i = w$. 
Then
\begin{align}
  b(\x, \v)  &= \sup_{\z \in \mathbb{S}(\x)} \sum_{i=1}^{n+1} z_{(i)}(v_{(i)} - v_{(i-1)}) \\
  &= \sup_{\z \in \mathbb{S}(\x)}  z_{(n+1)}(v_{(n+1)} - v_{(n)}) + z_{(1)}(v_{(1)} - v_{(0)}) \\
  &= \sup_{\z \in \mathbb{S}(\x)}  s_D (1 - v) + z_{(1)}(v - 0) \\
  &= \sup_{\z \in \mathbb{S}(\x)}  s_D  - (s_D - z_{(1)})  \frac{s_D - c_2}{s_D - \phi}\\
  &=  s_D  - (s_D - \phi)  \frac{s_D - c_2}{s_D - \phi} \text{ because $\frac{s_D - c_2}{s_D - \phi} \ge 0$}\\
  &= c_2.
\end{align}
Similarly, 
\begin{align}
  b(\x, \w)  &= \sup_{\z \in \mathbb{S}(\x)} \sum_{i=1}^{n+1} z_{(i)}(w_{(i)} - w_{(i-1)}) \\
  &= \sup_{\z \in \mathbb{S}(\x)}  z_{(n+1)}(w_{(n+1)} - w_{(n)}) + z_{(1)}(w_{(1)} - w_{(0)}) \\
  &= \sup_{\z \in \mathbb{S}(\x)}  s_D (1 - w) + z_{(1)}(w - 0) \\
  &= \sup_{\z \in \mathbb{S}(\x)}  s_D  - (s_D - z_{(1)})  \frac{s_D - c_1}{s_D - \phi}\\
  &=  s_D  - (s_D - \phi)  \frac{s_D - c_1}{s_D - \phi} \text{ because $\frac{s_D - c_1}{s_D - \phi} \ge 0$}\\
  &= c_1.
\end{align}

Since $b(\x, \u)$ is constructed from a linear function of $\u$ with non-positive coefficients, for any $\u$ such that $v \le  u_{(1)} \le \cdots \le u_{(n)} <  w $ we have:  
\begin{align}
    b(\x, \w) < b(\x, \u) \le b(\x, \v), 
\end{align}
which is equivalent to:
\begin{align}
    c_1 < b(\x, \u) \le c_2. 
\end{align}

So we have $v \le  u_{(1)} \le \cdots \le u_{(n)} <  w $ implies $c_1 < b(\x, \u) \le c_2$. Therefore for any $c_1, c_2$ such that $\phi \le c_1 < c_2 \le s_D$: 
\begin{align}
    &F_M(c_2) - F_M(c_1) \\
    &= \PP( c_1 < M \le c_2) \\
    &= \PP_{\U}( c_1 < b(\x, \U) \le c_2) \\
   &\ge \PP_{\U}( v \le  U_{(1)} \le \cdots \le U_{(n)} < w ) \\
   &= \PP_{\U}(~\forall i, 1 \le i \le n: v \le  U_i < w ) \\
   &= (w - v)^n \\
   &= \left (\frac{c_2 - c_1}{s_D - \phi} \right)^n\\
   &>0 \text{ because $c_1 < c_2$}. 
\end{align}
Since the support of $M$ is in $[\phi, s_D]$ we have that $F_M$ is strictly increasing on the support. 
\end{proof}

In summary, the Monte Carlo estimate of our bound will converge to the correct value as the number of samples grows.

Now we prove Theorem~\ref{thm:MC_final}. 
\begin{proof}[Proof of Theorem~\ref{thm:MC_final}]
To simplify the notation, we use $Q(\alpha)$ to denote $Q(\alpha,M)$. 
From Lemma~\ref{lem:MCMC}, since $F_M$ is strictly increasing on the support, for $\gamma$ such that $0 < \gamma \le \alpha$, $Q( 1 - \alpha) < Q(1 - \alpha + \gamma)$ and: 
\begin{align}
   \gamma = F( Q( 1 - \alpha), Q(1 - \alpha + \gamma)) \ge \left(\frac{Q(1 - \alpha + \gamma) - Q( 1 - \alpha) }{s_D - \sup_{\z \in \mathbb{S}(\x)} z_{(1)}} \right)^n.
\end{align}
Therefore, letting $\gamma = \min \left( \alpha, \left( \frac{\epsilon}{3 (s_D - \sup_{\z \in \mathbb{S}(\x))}z_{(1)})} \right)^n \right)$ we have that
\begin{align}
  Q(1 - \alpha + \gamma) &\le   Q( 1 - \alpha) + \gamma^{1/n} (s_D - \sup_{\z \in \mathbb{S}(\x)} z_{(1)}) \\
  &\le  Q( 1 - \alpha) + \epsilon/3. \label{eq:quantile_shift}
\end{align}
Let $p \definedas 1 - \alpha + \gamma$. From Lemma~\ref{lem:MC} and Lemma~\ref{lem:MCMC},
\begin{align}
\PP(   | \hat{q}_{\ell} - Q(p) | > \epsilon/3 ) \le 2 e^{-2l\delta} , 
\end{align}
where 
\begin{align}
\delta &= \min\left(p - F(Q(p) - \epsilon/3), F(Q(p) + \epsilon/3) - p\right) \\ 
&= \min\left(F(Q(p) - \epsilon/3, Q(p)), F(Q(p) , Q(p) + \epsilon/3) \right)\\
&\ge \left( \frac{\epsilon}{3(s_D - \sup_{\z \in \mathbb{S}(\x)} z_{(1)})} \right)^n.
\end{align}

Therefore letting $\ell = \left \lceil \frac{-\ln(\gamma/2)}{2} \left(\frac{3(s_D - \sup_{\z \in \mathbb{S}(\x)} z_{(1)})}{\epsilon}\right)^n \right \rceil$,
\begin{align}
    \PP(   | \hat{q}_{\ell} - Q(1-\alpha + \gamma) | > \epsilon/3 ) &\le 2 e^{-2l\left( \frac{\epsilon}{3(s_D - \sup_{\z \in \mathbb{S}(\x)} z_{(1)})} \right)^n }
    \\ &\le \gamma. 
\end{align}
Since $\PP(Q(1-\alpha + \gamma ) < \mu ) \le \alpha - \gamma$, using the union bound we have 
\begin{align}
\PP ( | \hat{q}_{\ell} - Q(1-\alpha + \gamma) | > \epsilon/3  \text{ OR } Q(1-\alpha + \gamma ) < \mu )  &\le \gamma + \alpha - \gamma \\ &= \alpha. 
\end{align}
And therefore, 
\begin{align}
1 - \alpha &\le \PP ( | \hat{q}_{\ell} - Q(1-\alpha + \gamma) | \le \epsilon/3  \text{ AND } Q(1-\alpha + \gamma ) \ge \mu )   \\
&\le \PP (  Q(1-\alpha + \gamma) \le \hat{q}_{\ell} +  \epsilon/3  \le Q(1-\alpha + \gamma) + 2\epsilon/3 \text{ AND } Q(1-\alpha + \gamma ) \ge\mu )  \\
&\le \PP (  \mu \le \hat{q}_{\ell} +  \epsilon/3   \le Q(1-\alpha + \gamma) + 2\epsilon/3  )  \\
&\le \PP (  \mu \le \hat{q}_{\ell} +  \epsilon/3  \le Q(1-\alpha ) + \epsilon  ) \text{ from Eq.~\ref{eq:quantile_shift}}. 
\end{align}
\end{proof}

\section{Discussion on Section~\ref{sec:comparison}}
\label{apx:comparison}
We discuss the case when the distribution is Bernoulli in Section~\ref{apx:bernoulli}, and present the proofs of Section~\ref{sec:vsAnderson} in Section~\ref{apx:proof_vsAnderson}. In Section~\ref{apx:equal_Anderson} we show that our bound is equal to Anderson's when $T$ is Anderson's bound and the lower bound of the support is $-\infty$, and could be better than Anderson's when $T$ is Anderson's bound and the lower bound of the support is finite and tight. 
\subsection{Special Case: Bernoulli Distribution}
\label{apx:bernoulli}
When we know that $D = \{0,1\}$, the distribution is Bernoulli. If we choose $T$ to be the sample mean, we will show that our bound becomes the same as the Clopper-Pearson confidence bound for binomial distributions~\citep{Clopper1934}.

If $\x, \z \in \{0,1\}^n$ and $T(\z) \le T(\x)$ then $m(\z, \u) \le m(\x,\u)$. Therefore for any $\u \in [0,1]^n$,
\begin{align}
b_{D,T}(\x,\u) = \sup_{\z\in \{0,1\}^n: T(\z) \le T(\x)} m_D(\z,\u)
= m_D(\x,\u).
\end{align}
Let $p_{\x}$ be the number of $0$'s in $\x$. Therefore the bound becomes the $1-\alpha$ quantile of $m_D(\x,\U)$ where
\begin{align}
m_D(\x,\U)
= 1 - \sum_{i=1}^n U_{(i)}(x_{(i+1)}- x_{(i)})
= 1 - U_{(p_{\x})}.
\end{align}
Therefore the bound is the $1-\alpha$ quantile of $1-U_{(p_{\x})}$. Then
\begin{align}
  \PP( U_{(p_{\x})} \le 1 - b^{\alpha}(\x) )  =  \PP( 1-U_{(p_{\x})} \ge b^{\alpha}(\x) ) = \alpha.
\end{align}
Let $\beta(i,j)$ denote a beta distribution with parameters $i$ and $j$. We use the fact that the order statistics of a uniform distribution are beta-distributed. Since $U_{(p_{\x})} \sim \beta(p_{\x}, n+1 - p_{\x})$, we have $1 - U_{(p_{\x})} \sim \beta(n - p_{\x} +1, p_{\x})$
\begin{align}
    b^{\alpha}(\x) =  Q(1 - \alpha,  \beta(n - p_{\x} +1, p_{\x})).
\end{align}
This is the same as the Clopper-Pearson upper confidence bound for binomial distributions.
\subsection{Proof of Section~\ref{sec:vsAnderson}}
\label{apx:proof_vsAnderson}
\begin{lemma}[Lemma~\ref{lem:induced_mean_bound}]
Let $\X = (X_1,  \cdots, X_n)$ be a sample of size $n$ from a distribution with mean $\mu$. Let $\l \in [0,1]^n$. If $G_{\X, \l}$ is a $(1-\alpha)$ lower confidence bound for the CDF then % $m(\X, \l)$ is a  upper confidence bound for $\mu$:
\begin{align}
    \PP_{\X}(m(\X, \l) \ge \mu) \ge 1 - \alpha. 
\end{align}
\end{lemma}

\begin{proof}
\footnote{The proof is implied in \citep{Anderson1969TechReport} but we provide it here for completeness} If $\forall y \in \mathcal{R}, F(y) \ge G_{\X, \l}(y)$ then
\begin{align}
    \forall i: 1\le i \le n, F(X_{(i)}) \ge \ell_{(i)}.
\end{align}
Recall that $m_D(\X, \l) =  s_D - \sum_{i=1}^n \ell_{(i)}(z_{(i+1)}- z_{(i)})$. Therefore if $\forall i: 1\le i \le n,~F(X_{(i)}) \ge \ell_{(i)}$ then
$m(\X, \l) \ge m(\X, F(\X))$. 

From Lemma~\ref{lem:helper2}, $m(\X, F(\X)) \ge \mu$. Therefore $m(\X, \l) \ge \mu$. And hence, finally, 
\begin{align}
    \PP(m(\X, \l) \ge \mu) &\ge \PP_{\X}\left(\forall y \in \mathcal{R}, F(y) \ge G_{\X, \l}(y) \right) \\ &= 1 - \alpha.
\end{align}
\end{proof}

We now show that if $G_{\X, \l}$ (Figure~\ref{fig:cdf1}) is a lower confidence bound, then the order statistics of $\l$ are element-wise smaller than the order statistics of a sample of size $n$ from the uniform distribution with high probability: 
\begin{lemma}
\label{lem:lower_bound_cdf}
Let $\U=U_1,...,U_n$ be a sample of size $n$ from the continuous uniform distribution on $[0,1]$. Let $\l \in [0,1]^n$ and $\alpha \in (0,1)$. If $\mathcal{D}^+$ is continous and $G_{\X, \l}$ is a $(1-\alpha)$ lower confidence bound for the CDF then:
\begin{align}
% \label{tor_bound}
    \PP_{\U} (~\forall i: 1 \le i \le n, U_{(i)} \ge \ell_{(i)}) \ge 1 - \alpha.
\end{align}
\end{lemma}
\begin{proof}
Let $K$ be the CDF of a distribution such that $K$ is continuous and strictly increasing on $\mathcal{D}^+$ (since $\mathcal{D}^+$ is continuous, $K$ exists). Let $\X = (X_1,  \cdots, X_n)$ be a sample of size $n$ from the distribution with CDF $K$. By Lemma~\ref{lem:CDF_vs_U}, $K(X)$ is uniformly distributed on $[0,1]$. 

By the definition of $G_{\X, \l}$, if $~\forall x \in C, K(y) \ge G_{\X, \l}(y)$ then: 
\begin{align}
      &K(y) \ge 0, && \text{ if }  y <X_{(1)} \\
      &K(y) \ge \ell_{(i)}, && \text{ if } X_{(i)} \le y < X_{(i+1)}\\
      &K(y) \ge 1, && \text{ if } y \ge s_C.
\end{align}
which is equivalent to:
\begin{align}
      ~\forall i: 1 \le i \le n, K(y) \ge \ell_{(i)}, & \text{ if } X_{(i)} \le y < X_{(i+1)}.
\end{align}
Since $K(y)$ is non-decreasing, this is equivalent to: 
\begin{align}
      ~\forall i: 1 \le i \le n, K(X_{(i)}) \ge \ell_{(i)}
\end{align}
Since $G_{\X, \l}$ is a lower confidence bound,
\begin{align}
1 - \alpha &\le 
      \PP_{\X} (~\forall y \in \mathcal{R}, K(y) \ge G_{\X, \l}(y)) \\
      &= \PP_{\X} (\forall i: 1 \le i \le n, K(X_{(i)}) \ge \ell_{(i)}) \\
      &= \PP_{\U} (\forall i: 1 \le i \le n, U_{(i)} \ge \ell_{(i)}) \text{ by Lemma~\ref{lem:CDF_vs_U}}. 
\end{align}
\end{proof}
To prove Theorem~\ref{thm:vsAnderson}, we prove the more general version where $G_{\X, \l}$ is a (possibly not exact) lower confidence bound for the CDF. 

\begin{theorem}
    \label{thm:LMTVsAnderson}
     Let $\l \in [0,1]^n$. 
    Let $D^+ = [-\infty, b]$. If $G_{\X, \l}$ is a  $1-\alpha$ lower confidence bound for the CDF, then for any sample size $n$, for all sample values $\x \in D^n$ and all $\alpha \in (0,1)$, using $T(\x) = m_{{D^+}}(\x, \l)$ to compute $b^\alpha_{{D^+},T} (\x)$ yields:
    \begin{equation}
        b^\alpha_{{D^+},T}(\x) \leq m_{{D^+}}(\x, \l).
    \end{equation}
\end{theorem}
\begin{proof}
Since $G_{\X, \l}$ is a  lower confidence bound for the CDF $F$, from Lemma~\ref{lem:lower_bound_cdf}, 
\begin{align}
    \PP(\forall i,  U_{(i)} \ge \ell_{(i)})  \ge 1- \alpha.
\end{align}
First we note that
\begin{align}
    b_{{D^+},T} (\x, \l)&= \sup_{\y: \y \in \mathbb{S}_{{D^+},T}(\x) } m_{{D^+}}(\y, \l) \\
    &= \sup_{m_{{D^+}}(\y , \l) \le  m_{{D^+}}(\x , \l)} m_{{D^+}}(\y, \l)  \\
    &= m_{{D^+}}(\x , \l).
\end{align}
Recall that $b_{{D^+},T}^\alpha(\x)$ is the $1-\alpha$ quantile of $b_{{D^+},T}(\x,\U)$. In order to show that $b_{{D^+},T}^\alpha(\x) \leq b_{{D^+},T}(\x, \l)$, we will show that
\begin{align}
    \PP(b_{{D^+},T}(\x,\U) \le b_{{D^+},T}(\x, \l)) \ge 1- \alpha.
\end{align}
Recall that $b_{{D^+},T}(\x,\U) = \sup_{\y \in \mathbb{S}_T(\x)} t_{{D^+}} - \sum_{i = 1}^{n} U_{(i)}(x_{(i+1)} - x_{(i)})$. Then if $~\forall i, U_{(i)} \ge \ell_{(i)}$ then $b_{{D^+},T}(\x,\U) \le b_{{D^+},T}(\x,\l)$. Therefore,

\begin{align}
    &\PP(b_{{D^+},T}(\x,\U) \le b_{{D^+},T}(\x, \l)) \\&\ge \PP(\forall i,  U_{(i)} \ge \ell_{(i)})  \\
   &\ge 1- \alpha, \text{ by Lemma~\ref{lem:lower_bound_cdf}}. 
\end{align}
\end{proof}

We can now show the comparison with Anderson's bound and Hoeffding's bound. 
\begin{theorem}[Theorem~\ref{thm:vsAnderson}]
Let $\l \in [0,1]^n$ be a vector such that  
$G_{\X, \l}$ is an exact $(1-\alpha)$ lower confidence bound for the CDF.

Let $D^+ = (-\infty, b]$. For any sample size $n$, for any sample value $\x \in D^n$,  for all $\alpha \in (0,1)$, using $T(\x) = b^{\alpha, \text{Anderson}}_{\l} (\x)$  yields
\begin{align}
   b^\alpha_{D^+,T}(\x) \le   b^{\alpha, \text{Anderson}}_{\l} (\x).
\end{align}
\end{theorem}
\begin{proof}
We have $b^{\alpha, \text{Anderson}}_{\l} (\x) = m_{D^+}(\x, \l)$ where $\l$ satisfies $G_{\X, \l}$ is a $1-\alpha$ lower confidence bound for the CDF. 
Therefore applying Theorem~\ref{thm:LMTVsAnderson} yields the result. 
\end{proof}

\begin{theorem}[Theorem~\ref{thm:vsHoeffding}]
Let $D^+ = (-\infty, b]$. For any sample size $n$, for any sample value $\x \in D^n$,for all $\alpha \in (0,0.5]$, using $T(\x) =  b^{\alpha, \text{Anderson}}_{\l} (\x)$  where $\l = \u^{\text{And}}$ yields
\begin{align}
   b^\alpha_{D^+,T}(\x) \le   b^{\alpha, \text{Hoeffding}} (\x),
\end{align}
where the inequality is strict when $n \ge 3$. 
\end{theorem}
\begin{proof}
The proof follows directly from Lemma~\ref{lem:AndersonvsHoeffding} and Theorem~\ref{thm:vsAnderson}.

Recall that $G_{\X, \u^\text{And}}$ is an exact $(1-\alpha)$ lower confidence bound for the CDF and therefore: 
\begin{align}
\PP_{\U} (\forall i: 1 \le i \le n, U_{(i)} \ge u^\text{And}_{(i)}) = 1 - \alpha.
\end{align}

From Theorem~\ref{thm:vsAnderson}, using $T(\x) = b^{\alpha, \text{Anderson}}_{\u^\text{And}} (\x)$  yields
\begin{align}
   b^\alpha_{D^+,T}(\x) \le   b^{\alpha, \text{Anderson}}_{\u^\text{And}} (\x). \label{eq:vsHoeffding1}
\end{align}

Let $\l \in [0,1]^n$ be defined such that
\begin{align}
    \ell_i\definedas \max \left \{0,i/n- \sqrt{\ln(1/\alpha)/(2n)} \right \}.
\end{align}

Since $G_{\X, \l}(x)$ is an $1-\alpha$ lower confidence bound, from Lemma~\ref{lem:lower_bound_cdf}: 
\begin{align}
    \PP_{\U} (\forall i: 1 \le i \le n, U_{(i)} \ge \ell_{(i)}) &\ge 1 - \alpha \\
    &= \PP_{\U} (\forall i: 1 \le i \le n, U_{(i)} \ge u^{\text{And}}_{(i)}).
\end{align}
Since $ u^\text{And}_i = \max \left \{0,i/n-\beta(n) \right \}$ and $\ell_i = \max \left \{0,i/n- \sqrt{\ln(1/\alpha)/(2n)} \right \}$, we have $\beta(n) \le \sqrt{\ln(1/\alpha)/(2n)}$, and therefore $u^\text{And}_i \ge \ell_i$ for all $i$. Therefore $m(\x, \l) \ge m(\x, \u^{\text{And}})$, i.e.
\begin{align}
    b^{\alpha, \text{Anderson}}_{\u^\text{And}} (\x)  \le b^{\alpha, \text{Anderson}}_{\l} (\x). \label{eq:vsHoeffding2}
\end{align}
From Eq.~\ref{eq:vsHoeffding1},  Eq.~\ref{eq:vsHoeffding2} and Lemma~\ref{lem:AndersonvsHoeffding} we have the result. 
\end{proof}

\subsection{Special Case: Reduction to Anderson's Bound}
\label{apx:equal_Anderson}

In this section we present a more detailed comparison to Anderson's. We show that our bound is equal to Anderson's when $T$ is Anderson's bound and the lower bound of the support is $-\infty$, and can be better than Anderson's when $T$ is Anderson's bound and the lower bound of the support is tight. 
\begin{theorem}
    \label{thm:equal_Anderson}
     Let $\l \in [0,1]^n$ be such that $\ell_i \ge 0~\forall i, 1 \le i \le n$. 
    Let $D^+ = [a, b]$. Let ${i_0} \definedas \arg\min_{i: \ell_{(i)} > 0} \ell_{(i)}$.  If $G_{\X, \l}$ is an exact $1-\alpha$ lower confidence bound for the CDF, then for any sample size $n$, for all sample values $\x \in D^n$ and all $\alpha \in (0,1)$, using $T(\x) = m_{{D^+}}(\x, \l)$ to compute $b^\alpha_{{D^+},T} (\x)$ yields:
    \begin{align}
        b^\alpha_{{D^+},T}(\x) = m_{{D^+}}(\x, \l) &\text{ if }  a \le b - \frac{b - m(\x, \l)}{\ell_{(i_0)}} \\
        b^\alpha_{{D^+},T}(\x) < m_{{D^+}}(\x, \l) &\text{ if }  a > b - \frac{b - m(\x, \l)}{\ell_{(i_0)}}  \text{ and } n > i_0.
    \end{align}
    In particular, if $x_i < b - \frac{(b-a)\ell_{(i_0)}}{\ell_{(n)}}~\forall i, 1\le i \le n$ and $n > i_0$ then $ b^\alpha_{{D^+},T}(\x) < m_{{D^+}}(\x, \l)$. 
\end{theorem}
\begin{proof}
From the proof of Theorem~\ref{thm:LMTVsAnderson}, if $U_i \ge \ell_i$ for all $i$, then $b_{{D^+},T}(\x,\U) \le m_{{D^+},T}(\x, \l)$ and therefore: 
\begin{align}
    \PP(b_{{D^+},T}(\x,\U) \le m_{{D^+},T}(\x, \l)) &\ge \PP \left(\cap_{i: 1\le i \le n}  \{U_i \ge \ell_i \} \right ) \\
    &= 1- \alpha 
\end{align}
 We will show that: 
\begin{align}
    \PP(b_{{D^+},T}(\x,\U) \le m_{{D^+},T}(\x, \l)) &\le 1- \alpha \text{ if } a \le \frac{b \ell_{(i_0)} - b + m(\x, \l)}{\ell_{(i_0)}}\\
    \PP(b_{{D^+},T}(\x,\U) \le m_{{D^+},T}(\x, \l)) &> 1- \alpha \text{ otherwise }, 
\end{align}
which implies: 
\begin{align}
    \PP(b_{{D^+},T}(\x,\U) \le m_{{D^+},T}(\x, \l)) &=1- \alpha \text{ if } a \le \frac{b \ell_{(i_0)} - b + m(\x, \l)}{\ell_{(i_0)}}\\
    \PP(b_{{D^+},T}(\x,\U) \le m_{{D^+},T}(\x, \l)) &> 1- \alpha \text{ otherwise } 
\end{align}
\begin{itemize}
    \item First we show that $ \PP(b_{{D^+},T}(\x,\U) \le m_{{D^+},T}(\x, \l)) \le 1- \alpha \text{ if } a \le \frac{b \ell_{(i_0)} - b + m(\x, \l)}{\ell_{(i_0)}}$. 
Recall that $b_{{D^+},T}(\x,\U) = \sup_{\y \in \mathbb{S}_{{D^+},T}(\x)} m_{{D^+},T}(\x,\U)$. We have: 
\begin{align}
    \PP(b_{{D^+},T}(\x,\U) \le m_{{D^+},T}(\x, \l))
    &= \PP \left( \sup_{\y \in \mathbb{S}_{{D^+},T}(\x)}  m_{{D^+},T}(\y,\U) \le m_{{D^+},T}(\x, \l)  \right )
\end{align}

Consider the set of points $\v^i$ of the form
\begin{align}
    &\v^i = (\underbrace{\gamma_i, \cdots, \gamma_i}_\text{$i$ times}, b, \cdots, b)
\end{align}
% or more precisely: 
% \begin{align}
%     & v^i_{(1)} = \cdots = v^i_{(i)} = \gamma_i \\
%     & v^i_{(i+1)} = \cdots = v^i_{(n)} = b
% \end{align}
where $\gamma_i$ satisfy $a \le \gamma_i \le b$ and  $ \sum_{i=1}^{n+1} v_{i}(\ell_{(i)} - \ell_{(i-1)}) = m(\x, \l)$, which is equivalent to:
\begin{align}
  & a \le \gamma_i \le b \\
  & \ell_{(i)} > 0 \\
  & \gamma_i = b -  \frac{ b - m(\x, \l) }{\ell_{(i)} }
\end{align}
Therefore if $a \le  b -  \frac{ b - m(\x, \l) }{\ell_{(i)} }$ for all $i$ such that $ \ell_{(i)} > 0$ then $\v^i \in \mathbb{S}_{{D^+},T}(\x)$ for all $i$ and: 
\begin{align}
    \PP(b_{{D^+},T}(\x,\U) \le m_{{D^+},T}(\x, \l))
    &\leq \PP \left( \cap_{i: \ell_{(i)} > 0} \{m_{{D^+},T}(\v^i,\U) \le m_{{D^+},T}(\x, \l) \} \right ) \\
     &= \PP \left(\cap_{i: \ell_{(i)} > 0}  \{b -  U_{(i)} (b - \gamma_i) \le  m_{{D^+},T}(\x, \l) \} \right ) \\
      &= \PP \left(\cap_{i: \ell_{(i)} > 0}  \{b -  U_{(i)} (b - (b -  \frac{ b - m(\x, \l) }{\ell_{(i)} })) \le  m_{{D^+},T}(\x, \l) \} \right ) \\
        &= \PP \left(\cap_{i: \ell_{(i)} > 0}  \{U_{(i)} \ge \ell_{(i)}  \} \right ) \\
        &= 1 - \alpha. 
\end{align}
Since $\ell_{(1)} \le \cdots \le \ell_{(n)}$, if $a \le  b -  \frac{ b - m(\x, \l) }{\ell_{(i_0)} }$ then  $a \le  b -  \frac{ b - m(\x, \l) }{\ell_{(i)} }$ for all $i$. Therefore if $a \le  b -  \frac{ b - m(\x, \l) }{\ell_{(i_0)} }$, then $    \PP(b_{{D^+},T}(\x,\U) \le m_{{D^+},T}(\x, \l)) = 1- \alpha$ and $ b^\alpha_{{D^+},T}(\x) = m_{{D^+}}(\x, \l)$. 
\item
Now we will show that if $a >  b -  \frac{ b - m(\x, \l) }{\ell_{(i_0)} }$ and $n > i_0$ then $\PP(b_{{D^+},T}(\x,\U) \le m_{{D^+},T}(\x, \l)) > 1- \alpha$. 

Let $\epsilon = \min \left( \inf_{\y \in \mathbb{S}_{{D^+},T}(\x)} \frac{(b - y_{(i_0 +1)})(1 - \ell_{(n)})}{b - a}, \ell_{(i_0)} \right)$. 

 We will show that if $a =  b -  \frac{ b - m(\x, \l) }{\ell_{(i_0)} } + \delta$ where $\delta >0$ then and $\ell_{(n)} > \ell_{(i_0)}$ and $b \ge y_{(i_0 +1)}  + \frac{\delta \ell_{(i_0 )} }{\ell_{(n)} - \ell_{(i_0)}}$ for all $\y \in \mathbb{S}_{{D^+},T}(\x)$ and therefore $\epsilon \ge \frac{\delta \ell_{(i_0 )} }{\ell_{(n)} - \ell_{(i_0)}} > 0$. Since $m(\y, \l) \le m(\x, \l)$ we have: 
\begin{align}
    &b(1 - \ell_{(n)}) + y_{(i_0+1)} (\ell_n - \ell_{(i_0)}) + (b -  \frac{ b - m(\x, \l) }{\ell_{(i_0)} } + \delta)\ell_{(i_0)} \\
    &\le b(1 - \ell_{(n)}) + y_{(i_0+1)} (\ell_n - \ell_{(i_0)}) + a \ell_{(i_0)} \\
    &\le b(1 - \ell_{(n)}) + y_{(i_0+1)} (\ell_n - \ell_{(i_0)}) + y_{(i_0)} \ell_{(i_0)} \\
    &\le m(\y, \l) \\
    &\le m(\x, \l). 
\end{align}
And therefore: 
\begin{align}
    &b(1 - \ell_{(n)}) + y_{(i_0+1)} (\ell_n - \ell_{(i_0)}) + (b -  \frac{ b - m(\x, \l) }{\ell_{(i_0)} } + \delta)\ell_{(i_0)} \le m(\x, \l) \\
    &\iff b(1 - \ell_{(n)}) +  y_{(i_0+1)} (\ell_n - \ell_{(i_0)}) + (b\ell_{(i_0)} + \delta \ell_{(i_0)} - b + m(\x, \l)) \le m(\x, \l) \\
     &\iff \delta \ell_{(i_0)}  \le (b - y_{(i_0+1)}) ( \ell_{(n)} - \ell_{(i_0)})
\end{align}
Therefore $\ell_{(n)} - \ell_{(i_0)} > 0$ and  $b \ge y_{(i_0 +1)}  + \frac{\delta \ell_{(i_0 )} }{\ell_{(n)} - \ell_{(i_0)}}$, and $\epsilon > 0$.

Let $\mathcal{U}_{\epsilon} = \{ \U: 0 \le U_j \le \epsilon ~ \forall j: 1 \le j < i_0,  \ell_{(i_0)}  - \epsilon \le U_{i_0} < \ell_{(i_0)}, 1- \epsilon \le U_{j} \le 1~ \forall j: i_0 < j \le n \}$. Since $\epsilon > 0$, $\PP(\mathcal{U}_{\epsilon}) > 0$. We will show that: 
\begin{align}
    &\PP(b_{{D^+},T}(\x,\U) \le m_{{D^+},T}(\x, \l)) \\ &\ge \PP \left(\cap_{i: 1\le i \le n}  \{U_{(i)}  \ge \ell_{(i)}  \} \right )  + \PP(\mathcal{U}_{\epsilon}) \label{eq:contains_u}\\
    &= 1 - \alpha + \PP(\mathcal{U}_{\epsilon}) \\
    &> 1-\alpha. 
\end{align}
We will show that if $\ell_{(i_0)} - \epsilon \le U_{(i_0)} \le \ell_{(i_0)}, 1- \epsilon \le U_{j} \le 1~ \forall j: i_0 < j \le n$ then $b_{{D^+},T}(\x,\U) \le m_{{D^+},T}(\x, \l)$. Then the set $\U$ satisfying $b_{{D^+},T}(\x,\U) \le m_{{D^+},T}(\x, \l)$ contains 2 disjoint sets $\mathcal{U}_{\epsilon}$ and the set $\U$ satisfying  $U_i \ge \ell_i$ for all $i$, which implies Eq.~\ref{eq:contains_u}. 

Let $\U'$ be such that $U'_{j} =0$ when $ 1 \le j < i_0$, $U'_{(i_0)} = \ell_{(i_0)} - \epsilon \ge 0$ (because $\epsilon \le \ell_{(i_0)}$) and $ U'_{(i)} = 1- \epsilon$ when $ i_0 < j \le n$. We will show that $m_{{D^+},T}(\y,\U') \le m_{{D^+},T}(\x, \l)$ for all $\y \in \mathbb{S}_{{D^+},T}(\x)$. 

We have: 
\begin{align}
    m_{{D^+},T}(\y,\U') &= b (1 -  U'_{(n)})+ \sum_{i=1}^n y_{(i)}(U'_{(i)} - U'_{(i-1)}) \\
    &= b\epsilon + y_{(i_0 +1)}( 1 - \epsilon - (\ell_{(i_0)} - \epsilon)) + y_{(i_0 )} (\ell_{(i_0)} - \epsilon) \\
    &= b\epsilon + y_{(i_0 +1)}( 1 - \ell_{(i_0)} ) + y_{(i_0 )} (\ell_{(i_0)} - \epsilon) \\
    &=( b- y_{(i_0 )})\epsilon + y_{(i_0 +1)}( 1 - \ell_{(i_0)} ) + y_{(i_0 )} \ell_{(i_0)} \\
    &\le ( b- y_{(i_0 )})\frac{(b - y_{(i_0 +1)})(1 - \ell_{(n)})}{b - a} + y_{(i_0 +1)}( 1 - \ell_{(i_0)} ) + y_{(i_0 )} \ell_{(i_0)} \\
    &\le ( b- y_{(i_0 )})\frac{(b - y_{(i_0 +1)})(1 - \ell_{(n)})}{b - y_{(i_0)}} + y_{(i_0 +1)}( 1 - \ell_{(i_0)} ) + y_{(i_0 )} \ell_{(i_0)} \\
    &=(b - y_{(i_0 +1)})(1 - \ell_{(n)})+ y_{(i_0 +1)}( 1 - \ell_{(i_0)} ) + y_{(i_0 )} \ell_{(i_0)} \\
    &=b (1 - \ell_{(n)})+ y_{(i_0 +1)}( \ell_{(n)} - \ell_{(i_0)} ) + y_{(i_0 )} \ell_{(i_0)} \\
    &=b (1 - \ell_{(n)})+ y_{(i_0 +1)} \sum_{i=i_0+1}^n ( \ell_{(i)} - \ell_{(i-1)} ) + y_{(i_0 )} \ell_{(i_0)} \\
     &\le b (1 - \ell_{(n)})+  \sum_{i=i_0+1}^n y_{(i)}( \ell_{(i)} - \ell_{(i-1)} ) + y_{(i_0 )} \ell_{(i_0)} \\
    &\le m(\y, \l) \\
    &\le m(\x, \l)
\end{align}
Since $\U'$ is the component-wise smallest element in $\mathcal{U}_{\epsilon}$ and $m(\x, \U)$ is a linear function of $\U$ with negative coefficient, we have $m_{{D^+},T}(\y,\U)  \le m_{{D^+},T}(\y,\U') \le m(\x, \l)$ for all $\U \in \mathcal{U}_{\epsilon}$. 

Note that if $x_i < b - \frac{(b-a)\ell_{(i_0)}}{\ell_{(n)}}~\forall i, 1\le i \le n$ then $a > b - \frac{b - m(\x, \l)}{\ell_{(i_0)}}$ and therefore if $n > i_0$ then  $b^\alpha_{{D^+},T}(\x) < m_{{D^+}}(\x, \l)$.
\end{itemize}
\end{proof}
For the specific case where $\l = \u^{And}$ we have the following result.
\begin{lemma}

Let $\l = \u^{And}$. Let $D^+ = [a, b]$. Let ${i_0}  \definedas \arg \min_{i: \ell_{(i)} > 0} \ell_{(i)}$. For any sample value $\x \in D^n$, for any sample size $n$ and for all $\alpha \in (0,1)$, using $T(\x) = b^{\alpha, \text{Anderson}}_{\l} (\x)$  yields:
\begin{align}
  b^\alpha_{D^+,T}(\x) 
    = b^{\alpha, \text{Anderson}}_{\l} (\x) &\text{ if } a \le \frac{b \ell_{(i_0)} - b + m(\x, \l)}{\ell_{(i_0)}} 
\end{align}

For any sample value $\x \in D^n$, for any sample size $n$ and for all $\alpha \in (0,1)$ satisfying $\frac{(n-1)^2}{n} > \frac{ln(1/\alpha)}{2}$ \footnote{ To satisfy this condition, when $\alpha = 0.01$, $n \ge 5$. When $\alpha = 0.05$, $n \ge 4$. When $\alpha = 0.1$, $n \ge 3$.}, using $T(\x) = b^{\alpha, \text{Anderson}} (\x)$  yields:
\begin{align}
  b^\alpha_{D^+,T}(\x)  < b^{\alpha, \text{Anderson}}_{\l} (\x) &\text{ if } a > \frac{b \ell_{(i_0)} - b + m(\x, \l)}{\ell_{(i_0)}} 
\end{align}
\end{lemma}
\begin{proof}
The proof follows from Theorem~\ref{thm:equal_Anderson}.

First we note that $\u^{And}$ satisfies:
    \begin{align}
        \PP_{\X} (~\forall x \in R, F(x) \ge G_{\X, \u^{And}}(x)) = 1 - \alpha.
    \end{align}
    
    We will now show that if $\frac{(n-1)^2}{n} > \frac{ln(1/\alpha)}{2}$, then $u^{And}_{(n-1)} > 0$ and therefore $ i_0 \le n-1$, which implies $n > i_0$.
    
  Using the Dvoretsky-Kiefer-Wolfowitz inequality~\citep{Dvoretzky1956} to define the $1-\alpha$ CDF lower bound via $\beta(n) = \sqrt{\ln(1/\alpha)/(2n)}$, we can compute a lower bound for $u^{And}_i$ as follows: 
    \begin{align}
        u^{And}_i \ge \max\left(0, i/n - \sqrt{\frac{\ln(1/\alpha)}{2n}} \right)
    \end{align}
    Therefore if $\frac{n-1}{n} - \sqrt{\frac{\ln(1/\alpha)}{2n})} > 0$ then $u^{And}_{(n-1)} > 0$. The condition $\frac{n-1}{n} - \sqrt{\frac{\ln(1/\alpha)}{2n}} > 0$ is equivalent to $\frac{(n-1)^2}{n} > \frac{\ln(1/\alpha)}{2}$. 
\end{proof}
\end{document}